\documentclass[12pt]{article}

\usepackage[a4paper, top=3cm, left=3cm, right=2cm, bottom=2cm]{geometry}
\usepackage[affil-it]{authblk}
\usepackage{amsmath, amsthm}
\usepackage{amssymb}
\usepackage{amsfonts}
\usepackage{dsfont}
\usepackage{enumerate}

    \newcommand{\Q}{\mathbb{Q}}

    \newcommand{\R}{\mathbb{R}}

    \newcommand{\dom}{\mbox{\rm dom}}

    \newcommand{\ran}{\mbox{\rm ran}}

    \newcommand{\thzfc}{\mathrm{ZFC}}

    \newcommand{\Awf}{\mathcal{A}}
    \newcommand{\Bwf}{\mathcal{B}}
    \newcommand{\Cwf}{\mathcal{C}}
    
    \newcommand{\Ewf}{\mathcal{E}}
    \newcommand{\Fwf}{\mathcal{F}}
    \newcommand{\Gwf}{\mathcal{G}}
    
    \newcommand{\Iwf}{\mathcal{I}}
    \newcommand{\Jwf}{\mathcal{J}}
    \newcommand{\Mwf}{\mathcal{M}}
    \newcommand{\Nwf}{\mathcal{N}}
    \newcommand{\Owf}{\mathcal{O}}
    \newcommand{\Pwf}{\mathcal{P}}

    \newcommand{\bfrak}{\mathfrak{b}}
    \newcommand{\cfrak}{\mathfrak{c}}

    \newcommand{\blocfrak}{\mathfrak{b}_\textrm{Loc}}

    \newcommand{\menos}{\smallsetminus}

    \newcommand{\frestr}{\!\!\upharpoonright\!\!}

    \newcommand{\add}{\mbox{\rm add}}

    \newcommand{\cof}{\mbox{\rm cof}}

    \newcommand{\Fin}{\mbox{\rm Fin}}

    \newcommand{\Bor}{\mathds{B}}
    \newcommand{\Cor}{\mathds{C}}
    \newcommand{\Dor}{\mathds{D}}
    
    \newcommand{\Loc}{\mathds{LOC}}
    
    \newcommand{\Por}{\mathds{P}}
    \newcommand{\Qor}{\mathds{Q}}

    \newcommand{\Qnm}{\dot{\mathds{Q}}}

    \newcommand{\imp}{{\ \mbox{$\Rightarrow$} \ }}

\title{Rothberger gaps in fragmented ideals}

\author{J\"org Brendle\thanks{Partially supported by Grant-in-Aid for Scientific Research (C)  24540126, Japan Society for the Promotion of Science.}}
\affil{Graduate School of System Informatics,\\
Kobe University,\\
1-1 Rokkodai, Nada-ku,\\
657-8501 Kobe, Japan.\\
E-mail: brendle@kurt.scitec.kobe-u.ac.jp}

\author{Diego Alejandro Mej\'ia\thanks{Supported by the Monbukagakusho (Ministry of Education, Culture, Sports, Science and Technology) Scholarship, Japan.}}
\affil{Graduate School of System Informatics,\\
Kobe University,\\
Kobe, Japan.\\
E-mail: damejiag@kurt.scitec.kobe-u.ac.jp}

\begin{document}

\makeatletter
\def\@roman#1{\romannumeral #1}
\makeatother

\makeatletter
\def\@maketitle{%
  \newpage
  \null
  \vskip 2em%
  \begin{center}%
  \let \footnote \thanks
    {\Large\bfseries \@title \par}%
    \vskip 1.5em%
    {\normalsize
      \lineskip .5em%
      \begin{tabular}[t]{c}%
        \@author
      \end{tabular}\par}%
    \vskip 1em%
    {\normalsize \@date}%
  \end{center}%
  \par
  \vskip 1.5em}
\makeatother

\newcounter{enuAlph}
\renewcommand{\theenuAlph}{\Alph{enuAlph}}

\theoremstyle{plain}
  \newtheorem{theorem}{Theorem}[section]
  \newtheorem{thm}[enuAlph]{Theorem}
  \newtheorem{corollary}[theorem]{Corollary}
  \newtheorem{lemma}[theorem]{Lemma}
  \newtheorem{prop}[theorem]{Proposition}
  \newtheorem{claim}[theorem]{Claim}
  \newtheorem{exer}{Exercise}
  \newtheorem{conjecture}[theorem]{Conjecture}
  \newtheorem{question}[theorem]{Question}
  \newtheorem{problem}[theorem]{Problem}
\theoremstyle{definition}
  \newtheorem{definition}[theorem]{Definition}
  \newtheorem{example}[theorem]{Example}
  \newtheorem{remark}[theorem]{Remark}
  \newtheorem{context}[theorem]{Context}

\maketitle


\renewcommand{\thefootnote}{}



\renewcommand{\thefootnote}{\arabic{footnote}}
\setcounter{footnote}{0}


\begin{abstract}
\noindent The~\emph{Rothberger number} $\bfrak (\Iwf)$ of a definable ideal $\Iwf$ on $\omega$ is the least cardinal $\kappa$ such that there exists a Rothberger gap of type $(\omega,\kappa)$ in the quotient algebra $\Pwf (\omega) / \Iwf$. We investigate $\bfrak (\Iwf)$ for a subclass of the $F_\sigma$ ideals, the fragmented ideals, and prove that for some of these ideals, like the linear growth ideal, the Rothberger number is $\aleph_1$ while for others, like the polynomial growth ideal, it is above the additivity of measure. We also show that it is consistent that there are infinitely many (even continuum many) different Rothberger numbers associated with fragmented ideals.
\end{abstract}

%
%

\section{Introduction}\label{Intro}

The investigation of gaps in the quotient Boolean algebra $\Pwf (\omega) / \mathrm{Fin}$ has a long and rich history. More than one hundred years ago, Hausdorff~\cite{hausdorff1,hausdorff2} constructed his celebrated $(\omega_1 , \omega_1)$-gap. Several decades later, Rothberger~\cite{rothb} produced an $(\omega,\bfrak)$-gap where $\bfrak$ denotes the \emph{(un)bounding number}, that is, the least size of an unbounded family in the preorder $(\omega^\omega , \leq^*)$ defined by $f \leq^* g$ iff $f(i) \leq g(i)$ holds for all but finitely many $i$. In fact, he proved that $\bfrak$ is the least cardinal $\kappa$ such that there are $(\omega,\kappa)$-gaps in $\Pwf (\omega) / \mathrm{Fin}$. It turns out that these are the only two types of gaps that exist in ZFC. Namely, not only under the continuum hypothesis $\mathrm{CH}$, but also under appropriate forcing axioms like the proper forcing axiom PFA\footnote{The argument in~\cite{todorpart} in fact shows that the conjunction of $\cfrak = \aleph_2$ and the open coloring axiom $\mathrm{OCA}$ is enough.}, any gap in $\Pwf (\omega) / \Fin$ is either of type $(\omega_1,\omega_1)$ or of type $(\omega, \bfrak)$~\cite[Theorem 8.6]{todorpart}.

Much more recently, research has shifted towards gaps in more general quotient algebras of the form $\Pwf (\omega) / \Iwf$ where $\Iwf$ is a definable ideal on the natural numbers. First, Mazur~\cite{mazur} showed that there are $(\omega_1,\omega_1)$-gaps in any quotient by an $F_\sigma$ ideal. Then, Todor\v cevi\'c~\cite{todorcevic} proved an important general result saying that for a large class of ideals $\Iwf$, including the $F_\sigma$ ideals and the analytic P-ideals, the gap spectrum of $\Pwf (\omega) / \Iwf$ includes the one of $\Pwf (\omega) / \mathrm{Fin}$, that is, every type of gap that exists in $\Pwf (\omega) / \mathrm{Fin}$ also exists in $\Pwf (\omega) / \Iwf$. Essentially, every Baire embedding $\Pwf(\omega)/\textrm{Fin}\to\Pwf(\omega)/\Iwf$ preserves $(\omega_1,\omega_1)$-gaps from $\Pwf (\omega) / \mathrm{Fin}$ when $\Iwf$ is analytic and it preserves all types of gaps when $\Iwf$ is an analytic P-ideal~\cite[Theorems 10 and 11]{todorcevic}. Moreover, for $F_\sigma$ ideals (actually for a larger class of ideals called \emph{Mazur ideals}), there exists a continuous embedding that preserves all gaps~\cite[Theorem 12]{todorcevic}. In particular, all quotients by $F_\sigma$ ideals or analytic P-ideals have $(\omega_1,\omega_1)$-gaps and $(\omega,\bfrak)$-gaps. Todor\v cevi\'c also addressed the general problem of determining the gap spectrum of such quotients~\cite[Problem 2]{todorcevic} (see~\cite[Section 5]{farah1} for a more detailed discussion of this problem).

This problem has triggered a number of interesting results. For example, Farah~\cite{farah2} proved that for all $F_\sigma$ P-ideals $\Iwf$ that are not generated by a single set over $\mathrm{Fin}$, there is an \emph{analytic Hausdorff gap} in $\Pwf (\omega) / \Iwf$. As a consequence, there is an $(\add (\Nwf), \add (\Nwf))$-gap in such quotients under the assumption $\add (\Nwf) = \cof (\Nwf)$ where $\add (\Nwf)$ and $\cof (\Nwf)$ are the \emph{additivity} (the smallest cardinality of a family of null sets whose union is not null) and the \emph{cofinality} of the ideal $\Nwf$ of Lebesgue null sets, respectively, and, in particular, a $(\cfrak,\cfrak)$-gap under Martin's axiom MA. This shows that the gap spectrum of such quotients may be larger than the one of $\Pwf (\omega) / \mathrm{Fin}$. More recently, Kankaanp\"a\"a~\cite{kankaanpaa} showed that there is an $(\omega , \add (\Mwf))$-gap in $\Pwf (\Q) / \mathrm{nwd}$ where $\mathrm{nwd}$ denotes the $F_{\sigma\delta}$ ideal of nowhere dense subsets of the rational numbers $\Q$ and $\add (\Mwf)$ is the \emph{additivity} of the meager ideal $\Mwf$ on the reals. In fact, similar to Rothberger gaps in $\Pwf (\omega) / \mathrm{Fin}$, $\add (\Mwf)$ is the least cardinal $\kappa$ such that there are $(\omega,\kappa)$-gaps in this quotient.

In this paper we investigate in greater depth for which uncountable definable cardinals $\kappa$ there are gaps of type $(\omega,\kappa)$ in quotients by definable ideals. In particular we will focus on the smallest cardinal $\kappa$ for which such gaps exist in a given quotient. Before outlining our main results, we review some basic notions and notation concerning gaps.

Given a Boolean algebra $\Bor$, $\Awf , \Bwf \subseteq \Bor$ are called \emph{orthogonal} if $a \land b = 0$ for all $a \in \Awf$ and $b \in \Bwf$. The pair $\langle\Awf,\Bwf\rangle$ is a \emph{gap} if there is no $c \in \Bor$ such that $a \land c = 0$ for all $a \in \Awf$ and $b \leq c$ for all $b \in \Bwf$. If both $\Awf$ and $\Bwf$ are $\sigma$-directed (i.e., given $\Cwf \subseteq \Awf$ countable there is $a \in \Awf$ such that $c \leq a$ for all $c \in \Cwf$, and similarly for $\Bwf$), we call $\langle\Awf , \Bwf\rangle$ a \emph{Hausdorff gap}. If, on the other hand, one of $\Awf$ and $\Bwf$ is countable, $\langle\Awf, \Bwf\rangle$ is a \emph{Rothberger gap}. $\langle\Awf ,\Bwf\rangle$ is a \emph{linear gap} of \emph{type $(\kappa,\lambda)$} (a \emph{linear $(\kappa,\lambda)$-gap}, for short) if $\Awf$ and $\Bwf$ are well-ordered of order type $\kappa$ and $\lambda$, respectively, i.e., $\Awf = \langle a_\alpha / \alpha <\kappa\rangle$ and $\Bwf = \langle b_\beta / \beta < \lambda\rangle$ are both strictly increasing. Obviously, a $(\kappa,\lambda)$-gap is Hausdorff if both $\kappa$ and $\lambda$ have uncountable cofinality, and Rothberger, if one of $\kappa$ and $\lambda$ is $\omega$.\footnote{By symmetry, it suffices to consider one of these two cases, and we shall always use the notation \emph{$(\omega,\lambda)$-gap} for Rothberger gaps.}

Let $\Iwf$ be an ideal\footnote{Unless we state the contrary, our ideals are \emph{non-trivial}, that is, they contain all finite sets but do not contain $\omega$.} on the natural numbers $\omega$. For $A,B\subseteq\omega$, $A\subseteq_\Iwf B$ means that $A\menos B$ belongs to $\Iwf$, and $\sim_\Iwf$ is the equivalence relation on $\Pwf(\omega)$ given by $A\sim_\Iwf B$ iff $A\subseteq_\Iwf B$ and $B\subseteq_\Iwf A$. $\Pwf(\omega)/\Iwf:=\Pwf(\omega)/\sim_\Iwf$ is the quotient Boolean algebra. For $A\subseteq \omega$, let $\Iwf\frestr A:=\{ X\in\Iwf\ /\ X\subseteq A\}$ be the \emph{restriction of the ideal $\Iwf$ to $A$}. Denote by $\textrm{Fin}$ the ideal of finite subsets of $\omega$ and let $A\subseteq^*B$ iff $A\subseteq_{\mathrm{Fin}} B$. Given a pointclass $\Gamma$ on the Cantor space $2^\omega$, an ideal $\Iwf$ is a \emph{$\Gamma$ ideal} if the set of characteristic functions of elements of $\Iwf$ belongs to $\Gamma$. The simplest non-trivial ideals are $F_\sigma$.

For an ideal $\Iwf$ on $\omega$ we define $\bfrak(\Iwf)$, the \emph{Rothberger number of $\Iwf$}, as the minimal cardinal $\kappa$ such that there exists an $(\omega,\kappa)$-gap in $\Pwf(\omega)/\Iwf$. Clearly, if $\bfrak(\Iwf)$ exists, it is a regular uncountable cardinal. Simpler equivalent ways to look at $\bfrak(\Iwf)$ are stated in Definition \ref{DefGaps} and Lemma \ref{Eqvb(I)}. Rothberger's result mentioned above says that $\bfrak(\textrm{Fin})=\bfrak$ (and this is our reason for using the letter $\bfrak$ for the Rothberger number) while Todor\v cevi\'c's theorem implies that $\bfrak(\Iwf)\leq\bfrak$ when $\Iwf$ is either an analytic P-ideal \emph{or} an $F_\sigma$ ideal. By Solecki's characterization~\cite{solecki,solecki2} of analytic P-ideals as ideals of the form $\mathrm{Exh}(\varphi)$ where $\varphi$ is a lower semicontinuous submeasure on $\Pwf (\omega)$ (see~\ref{idealmeas} below for details), it follows that $\bfrak(\Iwf) = \bfrak$ for such ideals.\footnote{We include a proof of this well-known fact in Remark~\ref{RemGaps}(7) because we could not find a reference.}

In our work, we concentrate on a class of $F_\sigma$ ideals introduced in work of Hru\v s\'ak, Rojas-Rebolledo, and Zapletal~\cite{hrusakzap}, namely the \emph{fragmented ideals} (see Definition \ref{DefFragId}). The reason for doing so is that on the one hand these ideals are combinatorially rather simple while on the other hand we obtain a rich spectrum of possible values for the Rothberger number for them. An important subclass are the \emph{gradually fragmented ideals} (\cite[Def. 2.1]{hrusakzap} and Definition \ref{DefFragId}). Typical examples are
\begin{itemize}
   \item the ideal $\mathcal{ED}_\textrm{fin}$ (see \cite[p.42]{hrusak}) whose underlying set consists of the ordered pairs below the identity function and which is generated by (graphs of) functions; this ideal is fragmented but not gradually fragmented (see also Example \ref{ExpFragId}(3));
   \item the \emph{linear growth ideal} $\Iwf_L$ (see \cite[p.56]{hrusak}) defined as follows: letting $\{a_i\}_{i<\omega}$ be the interval partition of $\omega$ such that $|a_i|=2^i$, say that $x\in\Iwf_L$ iff $\exists_{m<\omega}\forall_{i<\omega}(|x \cap a_i|\leq m\cdot(i+1))$; this ideal is fragmented but not gradually fragmented (see also Example \ref{ExpFragId}(4));
   \item the \emph{polynomial growth ideal} $\Iwf_P$ (see \cite[p.56]{hrusak}) given by $x\in\Iwf_P$ iff $\exists_{m<\omega}\forall_{i<\omega}(|x\cap a_i|\leq (\max\{i,2\})^m)$ where $\{a_i\}_{i<\omega}$ is the same interval partition of $\omega$ as for $\Iwf_L$; this ideal is gradually fragmented (see also Example \ref{ExpFragId}(2)).
\end{itemize}
Our main results are:

\begin{thm}
For a large class of fragmented, not gradually fragmented ideals $\Iwf$, including $\mathcal{ED}_\mathrm{fin}$ and $\Iwf_L$, we have $\bfrak (\Iwf) = \aleph_1$; i.e., there is an $(\omega,\omega_1)$-Rothberger gap in $\Pwf (\omega) / \Iwf$ (Theorems~\ref{b(EDfin)=aleph1}, \ref{b(unifsubmeas)=aleph1} and~\ref{b(meas)=aleph1}).
\end{thm}

\begin{thm}
\begin{enumerate}[(i)]
\item For all gradually fragmented ideals $\Iwf$,
\begin{itemize}
\item $\bfrak (\Iwf) \geq \add (\Nwf)$ (Corollary~\ref{b(gradfrag)aboveAddN}),
\item if $\Iwf$ is nowhere tall (see Definition \ref{Defpseudo-tall}), then $\bfrak (\Iwf) = \bfrak$ (Corollary~\ref{b(notpseudotall)=b}),
\item if $\Iwf$ is somewhere tall (see Definition \ref{Defpseudo-tall}), then $\bfrak (\Iwf) < \bfrak$ is consistent (Theorem~\ref{allb(I)belowb}).
\end{itemize}
\item For a large class of gradually fragmented ideals $\Iwf$ including $\Iwf_P$, $\bfrak (\Iwf) > \add (\Nwf)$ is consistent (Theorem~\ref{add(N)belowsomeb(I)}).
\item There may be (simultaneously) many gradually fragmented ideals with distinct Rothberger number (Theorem \ref{Consdiffb(I)}).
\end{enumerate}
\end{thm}

Theorem A provides the first examples of ``absolute" Rothberger gaps in definable quotients, that is, gaps that are \emph{not} forcing destructible like the Rothberger gaps in $\Pwf (\omega) / \mathrm{Fin}$ or $\Pwf (\omega) / \Iwf$ for analytic P-ideals and many other Borel ideals $\Iwf$. In that sense, these gaps in $\Pwf (\omega) / \mathcal{ED}_\mathrm{fin}$ or $\Pwf (\omega) / \Iwf_L$ are similar to Hausdorff's original gap in $\Pwf (\omega) / \mathrm{Fin}$. Theorem B in particular tells us that uncountably many (consistently distinct) definable cardinals are Rothberger numbers, and not just well-known cardinals like $\aleph_1$, $\bfrak$, or $\add (\Mwf)$.

This paper is organized as follows. In Section \ref{SecPrelim} we introduce the basic notions and results regarding ideals and forcing notions that we are going to use throughout the text. We prove several versions of Theorem A in Section~\ref{SecsmallRoth}. In Section \ref{SecDestroyGradFragGaps} we introduce a forcing notion that is crucial for destroying gaps of gradually fragmented ideals. Section \ref{SecPresProp} is devoted to the preservation properties related to the Rothberger number of a fixed tall ideal, which are fundamental for proving all our main consistency results in Section \ref{SecConsFrag} (see Theorem B). We discuss open questions related to our work in Section \ref{SecQ}.

%
%

\section{Preliminaries}\label{SecPrelim}

The following definition simplifies the notion of gaps for ideals on $\omega$.

\begin{definition}\label{DefGaps}
   Let $\Iwf$ be an ideal on $\omega$, $\Awf,\Bwf$ collections of subsets of $\omega$.
   \begin{enumerate}[(1)]
      \item The pair $\langle\Awf,\Bwf\rangle$ is $\Iwf$-orthogonal if $A\cap B\in\Iwf$ for all $A\in\Awf$ and $B\in\Bwf$.
      \item A subset $C$ of $\omega$ \emph{separates $\langle\Awf,\Bwf\rangle$} (with respect to $\Iwf$) if $A\cap C\in\Iwf$ for all $A\in\Awf$ and $B\subseteq_\Iwf C$ for all $B\in\Bwf$.
      \item The pair $\langle\Awf,\Bwf\rangle$ is an \emph{$\Iwf$-gap} (or a gap in $\Pwf(\omega)/\Iwf$) if it is $\Iwf$-orthogonal and no subset of $\omega$ separates it. When $|\Awf|=\kappa$ and $|\Bwf|=\lambda$, we say that the pair is an \emph{$\Iwf$-$(\kappa,\lambda)$-gap.} An $\Iwf$-$(\omega,\lambda)$-gap is called an \emph{$\Iwf$-Rothberger gap}.
      \item Denote the cardinal number $\bfrak(\Iwf)$ as the least cardinal number $\lambda$ such that there exists an $\Iwf$-$(\omega,\lambda)$-gap. We call this the \emph{Rothberger number of $\Iwf$.}
   \end{enumerate}
\end{definition}

To avoid inconsistencies with the notation, we use the term \emph{linear $(\kappa,\lambda)$-gap} for gaps of type $(\kappa,\lambda)$ (as given in the introduction), while \emph{$(\kappa,\lambda)$-gap} refers to Definition \ref{DefGaps}. However, as justified by the following result, it does not matter which notion of gap is used to define the Rothberger number of an ideal.

\begin{lemma}\label{Eqvb(I)}
   In the definition of $\bfrak(\Iwf)$ as the least $\lambda$ such that there exists an $\Iwf$-gap $\langle\Awf,\Bwf\rangle$ with $|\Awf|=\aleph_0$ and $|\Bwf|=\lambda$, the following restrictions on $\Awf$ and $\Bwf$ can be done without affecting the value of $\bfrak(\Iwf)$.
   \begin{enumerate}[(I)]
      \item $\Awf$ can either be
          \begin{enumerate}[(i)]
             \item a disjoint family, even a partition of $\omega$,
             \item a $\subseteq$-increasing sequence of length $\omega$, even with union equal to $\omega$, \emph{or}
             \item a $\subseteq$-increasing, $\subsetneq_\Iwf$-increasing sequence of length $\omega$, even with union equal to $\omega$.
          \end{enumerate}
          Moreover, it can be assumed that all the members of $\Awf$ are $\Iwf$-positive.
      \item $\Bwf$ can either be
          \begin{enumerate}[(i)]
             \item a $\subseteq_\Iwf$-increasing sequence of length $\lambda$ \emph{or}
             \item a $\subsetneq_\Iwf$ increasing sequence of length $\lambda$.
          \end{enumerate}
          Moreover, it can be assumed that all the members of $\Bwf$ are $\Iwf$-positive.
   \end{enumerate}
\end{lemma}

Many ideals on $\omega$ can be defined in terms of submeasures. Recall that, for a set $Y$, $\varphi:\Pwf(Y)\to[0,+\infty]$ is a \emph{submeasure on $\Pwf(Y)$} if $\varphi(\varnothing)=0$, $\varphi(x)<+\infty$ for any finite $x\subseteq Y$, and $\varphi$ is $\subseteq$-increasing and finitely subadditive, that is, $\varphi(A\cup B)\leq\varphi(A)+\varphi(B)$. If $Y=\omega$, a submeasure $\varphi$ on $\Pwf(\omega)$ is \emph{lower semicontinuous} if $\varphi(x)=\lim_{n\to+\infty}\varphi(x\cap n)$.

\begin{theorem}\label{idealmeas}
   Let $\Iwf$ be an ideal on $\omega$.
   \begin{enumerate}[(1)]
      \item \emph{(Mazur \cite{mazur})} $\Iwf$ is $F_\sigma$ iff there is a lower semicontinuous submeasure $\varphi$ on $\Pwf(\omega)$ such that $\Iwf=\mathrm{Fin}(\varphi):=\left\{ x\subseteq\omega\ /\ \varphi(x)<+\infty  \right\}$.
      \item \emph{(Solecki \cite{solecki,solecki2})} $\Iwf$ is an analytic P-ideal iff there is a lower semicontinuous submeasure $\varphi$ on $\Pwf(\omega)$ such that $\Iwf=\mathrm{Exh}(\varphi):=\{ x\subseteq\omega\ /\ \lim_{n\to\infty}\varphi(x\menos n)=0\}$. In particular, all analytic P-ideals are $F_{\sigma\delta}$.
   \end{enumerate}
\end{theorem}

\begin{remark}\label{RemGaps}
  \begin{enumerate}[(1)]
     \item The Rothberger number does not exist for maximal ideals because they have no gaps.
     \item Given $X$ an infinite subset of $\omega$, any $\Iwf\frestr X$-gap is an $\Iwf$-gap.
           Therefore, $\bfrak(\Iwf)\leq\bfrak(\Iwf\frestr X)$.
     \item (Hadamard \cite{hadamard}) For any ideal $\Iwf$ on $\omega$, there are no $\Iwf$-$(\omega,\omega)$-gaps. Therefore, if $\bfrak(\Iwf)$ exists, it is uncountable.
     \item If $\bfrak(\Iwf)$ exists, then it is regular.
     \item (Rothberger \cite{rothb}) There exists a $\Fin$-$(\omega,\bfrak)$-gap. Moreover, $\bfrak(\Fin)=\bfrak$.
     \item (Todor\v cevi\'c \cite{todorcevic}) If $\Iwf$ is an analytic P-ideal \emph{or} an $F_\sigma$-ideal, then $\bfrak(\Iwf)\leq\bfrak$.
     \item (Folklore) If $\Iwf$ is an analytic P-ideal, then $\bfrak(\Iwf)=\bfrak$. This follows from Solecki's characterization of analytic P-ideals (see Theorem \ref{idealmeas}(2)) and by an argument similar to the one for $\bfrak(\mathrm{Fin})\geq\bfrak$. Indeed, choose a lower semicontinuous submeasure $\varphi$ such that $\Iwf=\mathrm{Exh}(\varphi)$. Now, let $\langle\Awf,\Bwf\rangle$ be an $\Iwf$-orthogonal pair such that $\Awf=\{A_n\ /\ n<\omega\}$ is a partition of $\omega$ and $|\Bwf|<\bfrak$. Without loss of generality, we may assume that $\Iwf$ is an ideal in $\omega\times\omega$ and $A_n=\{n\}\times\omega$. In this notation, for $x\subseteq\omega\times\omega$, $x\in\mathrm{Exh}(\varphi)$ iff, for all $\epsilon>0$ there exists an $F\subseteq\omega\times\omega$ finite such that $\varphi(x\menos F)<\epsilon$. For $m<\omega$, we denote by $(A_n)_m:=\{(n,k)\in A_n\ /\ k<m\}$. For each $0<l<\omega$ and $B\in\Bwf$, let $g_{B,l}(n)$ be the minimal $m$ such that $\varphi((A_n\cap B)\menos(A_n)_m)<1/(l\cdot2^{n+1})$, which exists because $A_n\cap B\in\Iwf$. As $\{g_{B,l}\ /\ B\in\Bwf,\ 0<l<\omega\}$ has size $<\bfrak$, we can find $g\in\omega^\omega$ that dominates that set of functions. Put $C:=\bigcup_{n<\omega}(A_n)_{g(n)}$. Clearly, $A_n\cap C=(A_n)_{g(n)}\in\Iwf$ for every $n<\omega$, so it remains to show that $B\menos C\in\Iwf$ for any $B\in\Bwf$. Let $0<l<\omega$ and choose $N<\omega$ such that $g_{B,l}(n)\leq g(n)$ for every $n\geq N$. Let $F:=\bigcup_{n<N}(A_n)_{g_{B,l}(n)}$ and note that
         \[\varphi((B\menos C)\menos F)\leq\varphi\big(\bigcup_{n<\omega}(A_n\cap B\menos(A_n)_{g_{B,l}(n)})\big)\leq\sum_{n<\omega}\frac{1}{l\cdot2^{n+1}}=\frac{1}{l},\]
         where the last inequality holds because of the lower semicontinuity of the submeasure.
  \end{enumerate}
\end{remark}

We introduce the following notation. The quantifiers $\forall^\infty$ and $\exists^\infty$ mean, respectively, ``for all but finitely many'' and ``there are infinitely many'', where the index of the quantifier varies over a countable set. $id_\omega$ is the identity function from $\omega$ to $\omega$. For $f,g\in\omega^\omega$ and $c<\omega$, we extend the use of the notation for operations with natural numbers to functions, that is, $f\cdot g$ is the function such that $(f\cdot g)(i)=f(i)\cdot g(i)$, $(f^g)(i)=f(i)^{g(i)}$, $(cf)(i)=c\cdot f(i)$, etc. We may use this notation for real valued functions as well. Also, natural numbers may represent constant functions, that is, a natural number $n$ may represent the constant function from $\omega$ to $\{n\}$. This will be clear from the context.

We define a particular case of $F_\sigma$-ideals that will fit our purposes to obtain, consistently, ideals with Rothberger number strictly below $\bfrak$.

\begin{definition}[{\cite[Def. 2.1]{hrusakzap}}]\label{DefFragId}
  \begin{enumerate}[(1)]
     \item An ideal $\Iwf$ is \emph{fragmented} if there exists a partition $\{a_i\}_{i<\omega}$ of $\omega$ into non-empty finite sets and, for each $i<\omega$, a submeasure $\varphi_i:\Pwf(a_i)\to[0,+\infty)$ such that $x\in\Iwf$ iff $\{\varphi_i(x\cap a_i)\}_{i<\omega}$ is bounded (in $[0,+\infty)$). In this case, we say that $\Iwf=\Iwf\langle a_i,\varphi_i\rangle_{i<\omega}$. Writing $\bar{\varphi}(x)=\sup_{i<\omega}\{\varphi_i(x\cap a_i)\}$, $\bar{\varphi}$ turns out to be a lower semicontinuous submeasure on $\Pwf(\omega)$ with $\Iwf=\mathrm{Fin}(\bar{\varphi})$. Thus, any fragmented ideal is $F_\sigma$.
     \item A fragmented ideal $\Iwf=\Iwf\langle a_i,\varphi_i\rangle_{i<\omega}$ is \emph{gradually fragmented} if, for any $k<\omega$, there exists an $m\in\omega$ such that
         \[\forall_{l<\omega}\forall^\infty_{i<\omega}\forall_{B\subseteq\Pwf(a_i)}\big[\big(|B|\leq l\textrm{\ and }\forall_{b\in B}(\varphi_i(b)\leq k)\big)\imp\varphi_i(\bigcup B)\leq m\big].\]
         In this case, a function $f:\omega\to\omega$ \emph{witnesses the gradual fragmentation of $\Iwf$} if, for any $k<\omega$, $f(k)$ satisfies the same property as $m$ above.
  \end{enumerate}
\end{definition}

\begin{remark}\label{RemFragId}
  \begin{enumerate}[(1)]
     \item $|B|\leq l$ can be replaced by $|B|=l$ in the equation that describes gradually fragmented ideal. Also, $B$ can be restricted to pairwise disjoint families.
     \item A dichotomy proved in \cite[Thm 2.4]{hrusakzap} implies that the gradual fragmentation of a fragmented ideal does not depend on the partition and the sequence of submeasures that witness the fragmentation.
     \item If $\varphi:\Pwf(Y)\to[0,+\infty]$ is a submeasure, then $\varphi'(x)=\lceil\varphi(x)\rceil$ (least integer above $\varphi(x)$ if it is $<+\infty$, or else, it is $+\infty$) is also a submeasure. Therefore, if $\Iwf=\Iwf\langle a_i,\varphi_i\rangle_{i<\omega}$ is a fragmented ideal, with $\varphi'_i:\Pwf(a_i)\to\omega$ such that $\varphi'_i(x)=\lceil\varphi_i(x)\rceil$, we get that $\Iwf=\Iwf\langle a_i,\varphi'_i\rangle_{i<\omega}$.
     \item If $c$ is a positive real, then $\Iwf\langle a_i,\varphi_i\rangle_{i<\omega}=\Iwf\langle a_i,c\varphi_i\rangle_{i<\omega}$.
     \item The previous facts imply that, if $\Iwf$ is a fragmented ideal but is not gradually fragmented, then we can find $\langle a_i,\varphi_i\rangle_{i<\omega}$ such that $\Iwf=\Iwf\langle a_i,\varphi_i\rangle_{i<\omega}$ and, for all $m<\omega$, there exists an $l<\omega$ such that
         \[\exists^\infty_{i<\omega}\exists_{B\subseteq\Pwf(a_i)}\big(|B|=l\textrm{, }\forall_{b\in B}(\varphi_i(b)=1)\textrm{\ and\ }\varphi_i(\bigcup B)>m\big).\]
     \item A fragmented ideal may be trivial, e.g., choose any partition of $\omega$ into non-empty finite sets and use the zero-measure in each piece of the partition. The trivial ideal clearly is gradually fragmented. A fragmented ideal $\Iwf\langle a_i,\varphi_i\rangle_{i<\omega}$ is not trivial
         iff $\forall_{k<\omega}\exists^\infty_{i<\omega}(\varphi_i(a_i)>k)$.
  \end{enumerate}
\end{remark}

\begin{example}\label{ExpFragId}
  \begin{enumerate}[(1)]
    \item Given a finite set $Y$ and a real number $c\geq2$, $\varphi(x)=\log_c(|x|+1)$ defines a submeasure on $\Pwf(Y)$.
    \item Given $c\in\omega^\omega$, $c\geq2$ that converges to infinity and any partition $P=\{a_i\}_{i<\omega}$ of $\omega$ into non-empty finite sets, denote by $\Iwf_c(P):=\Iwf\langle a_i,\varphi_i\rangle_{i<\omega}$ where $\varphi_i(x)=\log_{c(i)}(|x|+1)$ for $x\subseteq a_i$. In view of Remark \ref{RemFragId}(3), $\varphi_i(x)$ can also be defined as the least $k<\omega$ such that $|x|<c(i)^k$.

        This ideal is gradually fragmented. Indeed, $f:\omega\to\omega$, $f(k)=k+1$ witnesses the gradual fragmentation of the ideal, as
        \begin{multline*}
          \forall_{l<\omega}\forall_{i,c(i)\geq l}\forall_{B\subseteq\Pwf(a_i)}\big[\big(|B|\leq l\textrm{\ and\ }\forall_{x\in B}(|x|< c(i)^k)\big)\\
          \imp|\bigcup B|< c(i)^{k+1}\big]
        \end{multline*}
        This ideal is not trivial iff $\forall_{k<\omega}\exists^\infty_{i<\omega}(|a_i|\geq c(i)^k)$.

        This is a generalization of the polynomial growth ideal $\Iwf_P$, which is
        $\Iwf_c(\{ a_i\}_{i<\omega})$ where $\{a_i\}_{i<\omega}$ is the interval partition of $\omega$ such that $|a_i|=2^i$ and $c=\max\{id_\omega,2\}$.
    \item An equivalent definition of the ideal $\mathcal{ED}_\textrm{fin}$ mentioned in the Introduction is given by $\mathcal{ED}_\textrm{fin}=\Iwf\langle a_i,\varphi_i\rangle_{i<\omega}$ where $\{a_i\}_{i<\omega}$ is the interval partition such that $|a_i|=i+1$ and $\varphi_i(x)=|x|$ for $x\subseteq a_i$. To see that it is not gradually fragmented note that, for every $m<\omega$, $l>m$ and $i\geq l$, if $B\subseteq\Pwf(a_i)$ is a disjoint family of size $l$ and $\forall_{x\in B}(|x|=1)$, then $|\bigcup B|=l>m$.
    \item Let $g:\omega\to\omega\menos\{0\}$ and $\{a_i\}_{i<\omega}$ a partition of $\omega$ into non-empty finite sets. Define $\Iwf=\Iwf\langle a_i,\varphi_i\rangle_{i<\omega}$ where $\varphi_i(x)=|x|/g(i)$.
        If the ideal is non-trivial, that is, the sequence of reals $\{|a_i|/g(i)\}_{i<\omega}$ is not bounded, then $\Iwf$ is not gradually fragmented. Indeed, for $m<\omega$, $l>m$ and those $i$ such that $|a_i|/g(i)\geq l$ (there are infinitely many such $i$), whenever $B\subseteq\Pwf(a_i)$ is a disjoint family of size $l$ such that $\forall_{b\in B}(|b|=g(i))$, then $|\bigcup B|/g(i)= l>m$. It is clear that $\mathcal{ED}_\textrm{fin}$ is a particular case of this ideal. Also, the linear growth ideal $\Iwf_L$ is a particular case with $\{a_i\}_{i<\omega}$ the interval partition of $\omega$ such that $|a_i|=2^i$ and $g(i)=i+1$.
    \item Let $g:\omega\to\omega\menos\{0\}$ and $\{a_i\}_{i<\omega}$ a partition of $\omega$ into non-empty finite sets. Define $\Iwf=\Iwf\langle a_i,\varphi_i\rangle_{i<\omega}$ where $\varphi_i(x)=|x|^{1/g(i)}$. Then, $\Iwf$ is gradually fragmented iff $\exists_{m<\omega}\forall_{l<\omega}\forall^\infty_{i<\omega}(\min\{l,|a_i|\}\leq m^{g(i)})$. To prove this first note that, in the case where such $m$ exists, $f(k)=m\cdot k$ witnesses the gradual fragmentation of $\Iwf$. Indeed, for $l<\omega$, let $N<\omega$ be such that $\forall_{i\geq N}(\min\{l,|a_i|\}\leq m^{g(i)})$ so, for $i\geq N$ and $B\subseteq\Pwf(a_i)$ of size $\leq l$ such that all its members have size $\leq k^{g(i)}$, $|\bigcup B|\leq(m\cdot k)^{g(i)}$.

        For the other direction, assume that $\forall_{m<\omega}\exists_{l<\omega}\exists^\infty_{i<\omega}(m^{g(i)}<\min\{l,|a_i|\})$. For $m<\omega$ choose $l<\omega$ and $W\subseteq\omega$ infinite such that $m^{g(i)}<\min\{l,|a_i|\}$ for all $i\in W$. Then, for any $B\subseteq\Pwf(a_i)$ of size $m^{g(i)}+1$ whose members are singletons (such a family exists), $|\bigcup B|>m^{g(i)}$.
  \end{enumerate}
\end{example}

The following discussion about tallness for fragmented ideals will be relevant for certain consistency results and characterizations of these ideals. Recall that an ideal $\Iwf$ on $\omega$ is \emph{tall} if, for every $X\in[\omega]^\omega$, there is an infinite set in $\Iwf\frestr X$. Note that the ideals of Example \ref{ExpFragId} are tall.

\begin{lemma}\label{tallfrag}
   Let $\Iwf=\Iwf\langle a_i,\varphi_i\rangle_{i<\omega}$ be a fragmented ideal. The following are equivalent.
   \begin{enumerate}[(i)]
      \item $\Iwf$ is tall.
      \item There exists a $k$ such that, for every $i<\omega$ and $j\in a_i$, $\varphi_i(\{j\})\leq k$.
      \item $\forall_{m<\omega}\exists_{l>m}\forall_{i<\omega}\forall_{x\subseteq a_i}(\varphi_i(x)>m\imp \exists_{x'\subseteq x}(m<\varphi_i(x')\leq l))$.
      \item The previous formula but with $m=0$.
      \item The formula of (iii) with $\exists_{m<\omega}$ instead of the universal quantifier.
   \end{enumerate}
\end{lemma}

\begin{proof}
   To see (i) implies (ii), assume the negation of (ii). Therefore, we can find $W:=\{j_k\ /\ k<\omega\}\subseteq\omega$ such that $\bar{\varphi}(\{j_k\})>k$ for any $k<\omega$. Then, it is clear that $\Iwf\frestr W$ does not contain infinite sets.

   Assume (ii) to prove (iii). Let $k>0$ be as in (ii). Now, for $m<\omega$, $l=m+k$ works. By contradiction, assume that there are $i<\omega$ and $x\subseteq a_i$ such that $\varphi_i(x)>m$ and all its subsets have submeasure not in $(m,m+k]$. In particular,
   $\varphi_i(x)>m+k$. When extracting one point of $x$, its submeasure is still bigger than $m$ and, then, bigger than $m+k$. By repeating this process, we get $\varphi_i(\varnothing)>m+k$ at the end, which is a contradiction.

   To finish, we prove (v) implies (i). Let $m$ and $l>m$ be as in (v) and assume that $W\subseteq\omega$ is infinite. Now, for each $i<\omega$, if $\varphi_i(W\cap a_i)>m$ then there exists a $y_i\subseteq W\cap a_i$ with submeasure in $(m,l]$. If $\varphi_i(W\cap a_i)\leq m$, put $y_i=W\cap a_i$. Then, $y:=\bigcup_{i<\omega}y_i\subseteq W$ is infinite and $\bar{\varphi}(y)\leq l$, so $y\in\Iwf\frestr W$.
\end{proof}

\begin{corollary}\label{tallfragnotP-ideal}
   Any non-trivial tall fragmented ideal is not a P-ideal.
\end{corollary}

\begin{proof}
   Let $\Iwf=\Iwf\langle a_i,\varphi_i\rangle_{i<\omega}$ be a non-trivial tall fragmented ideal. Find $L\subseteq\omega$ infinite such that $\{\varphi_i(a_i)\}_{i\in L}$ is strictly increasing. Choose $\{l_k\}_{k<\omega}$ strictly increasing by applying, recursively, Lemma \ref{tallfrag}(iii) and starting with $l_{0}=0$. Also, construct a strictly increasing sequence $\{N_k\}_{k<\omega}$ of natural numbers such that
   $\varphi_i(a_i)>\sum_{j\leq k}l_j+l_k$ for all $i\geq N_k$, $i\in L$. By recursion on $k$, choose $x^k_i\subseteq a_i$ for $i\geq N_k$, $i\in L$ such that $l_k<\varphi_i(x^k_i)\leq l_{k+1}$ and $x^k_i\cap x^j_i=\varnothing$ for each $j<k$. Indeed, for $i\geq N_k$, $\varphi_i(\bigcup_{j<k}x_i^j)\leq\sum_{j\leq k}l_j$, so its complement with respect to $a_i$ has
   submeasure bigger than $l_k$. Thus, by Lemma \ref{tallfrag}(iii), there exists an $x_i^k\subseteq a_i\menos\bigcup_{j<k}x_i^j$ as required.

   Put $x^k:=\bigcup\{ x^k_i\ /\ i\geq N_k\textrm{\ and\ }i\in L\}$, which is clearly in $\Iwf$. Now, if $y\subseteq\omega$ is such that $x^k\subseteq^* y$ for all $k<\omega$, we get that $l_k<\varphi_{i_k}(x^k_{i_k})$ for some $i_k\in L$ such that $x^k_{i_k}\subseteq y$. Therefore, $y\notin\Iwf$.
\end{proof}

The following notion is relevant for characterizing the fragmented ideals that can, consistently, have Rothberger number strictly less than $\bfrak$.

\begin{definition}\label{Defpseudo-tall}
   An ideal $\Iwf$ on $\omega$ is \emph{somewhere tall} if there exists an $\Iwf$-positive $X\subseteq\omega$ such that
   $\Iwf\frestr X$ is tall. An ideal is \emph{nowhere tall} if it is not somewhere tall.
\end{definition}

Corollary \ref{tallfragnotP-ideal} implies directly that any somewhere tall fragmented ideal is not a P-ideal. On the other hand, nowhere tall fragmented ideals can be simply characterized. For example, a fragmented ideal $\Iwf=\Iwf\langle a_i,\varphi_i\rangle_{i<\omega}$ where $\{|a_i|\}_{i<\omega}$ is bounded is nowhere tall. Indeed, if $X\subseteq\omega$ is $\Iwf$-positive then it is necessary that $\{\bar{\varphi}(\{j\})\}_{j\in X}$ is unbounded and, by Lemma \ref{tallfrag}, $\Iwf\frestr X$ is not tall. A converse of this and the mentioned characterization is stated as follows.

\begin{lemma}\label{fragnotpseudo-tallchar}
   If $\Iwf$ is a nowhere tall fragmented ideal on $\omega$, then $\Iwf=\Iwf\langle a_i,\varphi_i\rangle_{i<\omega}$ where $a_i=\{i\}$. Moreover, $\Iwf$ is gradually fragmented and can only be one of the following ideals:
   \begin{enumerate}[(i)]
      \item $\Iwf=\{ x\subseteq\omega\ /\ x\subseteq^* A \}$ for some $A\subseteq\omega$, or
      \item $\Iwf$ is the ideal generated by some infinite partition of $\omega$ into infinite sets.
   \end{enumerate}
\end{lemma}

\begin{proof}
   Let $\Iwf=\Iwf\langle a'_i,\varphi'_i\rangle_{i<\omega}$. By Lemma \ref{tallfrag}(ii), $\Iwf$ nowhere tall means that, for any $\Iwf$-positive $X\subseteq\omega$, $\{\bar{\varphi}'(\{j\})\}_{j\in X}$ is unbounded. Therefore, for any $x\subseteq\omega$, $x\in\Iwf$ iff $\{\bar{\varphi}'(\{j\})\}_{j\in x}$ is bounded, so $\Iwf=\langle a_i,\varphi_i\rangle_{i<\omega}$ where $a_i=\{i\}$ and $\varphi_i(\{i\})=\lceil\bar{\varphi}'(\{i\})\rceil$. Here, the identity function witnesses the gradual fragmentation of $\Iwf$. Now, for $m<\omega$, put $I_m:=\{i<\omega\ /\ \varphi_i(\{i\})=m\}$. Note that $\{I_m\}_{m<\omega}$ is a partition of $\omega$ and that $\Iwf$ is generated by this partition. If $I_m$ is infinite for infinitely many $m<\omega$, then we easily get (ii). Otherwise, if there is some $N<\omega$ such that $I_m$ is finite for all $m\geq N$, then we get (i) with $A:=\bigcup_{m<N}I_m$.
\end{proof}

Note that the case (i) gives us a P-ideal, so $\bfrak(\Iwf)=\bfrak$ for such non-trivial $\Iwf$ by Remark \ref{RemGaps}(7). In case (ii), $\Iwf$ is not a P-ideal but we are going to prove in Corollary~\ref{b(notpseudotall)=b} that still $\bfrak(\Iwf)=\bfrak$. In Section \ref{SecConsFrag} we prove that every somewhere tall ideal has, consistently, Rothberger number strictly less than $\bfrak$.

In Sections \ref{SecDestroyGradFragGaps}, \ref{SecPresProp} and \ref{SecConsFrag}, we are going to look at fragmented ideals from the forcing point of view. Note that a fragmented ideal $\Iwf=\Iwf\langle a_i,\varphi_i\rangle_{i<\omega}$ is coded by the real $\langle a_i,\varphi_i\rangle_{i<\omega}$, so the formula $x\in\Iwf$ is clearly $F_\sigma$ and expressions like ``$\Iwf$ is gradually fragmented" and ``$\Iwf$ is tall" (see Lemma \ref{tallfrag}) are arithmetical and, therefore, absolute notions.

We conclude this section with a short review of the known forcing notions that we are going to use throughout this paper. Recall \emph{Hechler forcing} $\Dor$ as the poset whose conditions are ordered pairs $(s,f)$ where $s\in\omega^{<\omega}$, $f\in\omega^\omega$ and $s\subseteq f$. Its order is given by $(t,g)\leq(s,f)$ iff $s\subseteq t$ and $f\leq g$. Clearly, $\Dor$ adds a dominating real over the ground model, that is, a real in $\omega^\omega$ that is an $\leq^*$-upper bound of the ground model reals.

For a function $h\in\omega^\omega$, denote by $S(\omega,h):=\prod_{i<\omega}[\omega]^{\leq h(i)}$, $S_n(\omega,h):=\prod_{i<n}[\omega]^{\leq h(i)}$ and $S_{<\omega}(\omega,h):=\bigcup_{n<\omega}S_n(\omega,h)$. For $x\in\omega^\omega$ and $\psi:\omega\to[\omega]^{<\omega}$, define the relation $x\in^*\psi$ iff $\forall^\infty_{i<\omega}(x(i)\in \psi(i))$. We often refer to functions from $\omega$ to $[\omega]^{<\omega}$ as \emph{slaloms}. If $h$ is a non-decreasing function that converges to infinity, define $\Loc^h$, the \emph{localization forcing for $h$}, as the poset with conditions $(s,F)$ where $s\in S_{<\omega}(\omega,h)$ and $F$ is a finite subset of $\omega^\omega$ such that $|F|\leq h(|s|)$. The order is given by $(s',F')\leq(s,F)$ iff $s\subseteq s'$, $F\subseteq F'$ and, for every $i\in|s'|\menos|s|$, $\{f(i)\ /\ f\in F\}\subseteq s'(i)$. This forcing adds a slalom in $S(\omega,h)$ that $\in^*$-covers all the reals in $\omega^\omega$ of the ground model. This forcing is useful to increase $\add(\Nwf)$ because

\begin{theorem}[{\cite[Thm. 2.3.9]{barju}} Bartoszy\'nski's characterization of $\add(\Nwf)$]\label{BartcharAdd(N)}
   If $h\in\omega^\omega$ converges to infinity, then $\add(\Nwf)$ is the least size of a family $\Fwf\subseteq\omega^\omega$ that cannot be $\in^*$-covered by a single slalom in $S(\omega,h)$.
\end{theorem}

Recall the following strengthenings of the countable chain condition (ccc) for posets. Let $\Por$ be a forcing notion. For $n<\omega$, a subset $P\subseteq\Por$ is \emph{$n$-linked} if, for any $F\subseteq P$ with $|F|\leq n$, there exists a $p\in\Por$ that extends all the members of $F$. We say that $P$ is \emph{centered} if it is $n$-linked for all $n<\omega$. For an infinite cardinal number $\mu$, say that $\Por$ is \emph{$\mu$-linked} if it is equal to a union of $\leq\mu$-many 2-linked subsets. Likewise, $\Por$ is \emph{$\mu$-centered} if it is equal to a union of $\leq\mu$-many centered subsets. For $\mu=\aleph_0$, it is usual to say $\sigma$-linked and $\sigma$-centered, respectively. Clearly, $\mu$-centeredness implies $\mu$-linkedness and $\sigma$-linkedness implies ccc, moreover, it implies the \emph{Knaster condition}, which says that every uncountable subset of $\Por$ has an uncountable 2-linked subset. It is easy to verify that $\Dor$ is $\sigma$-centered and $\Loc^h$ is $\sigma$-linked.

%
%

\section{Ideals with small Rothberger number}\label{SecsmallRoth}

In this section, we present a wide class of fragmented not gradually fragmented ideals that have, provably in $\thzfc$, Rothberger number equal to $\aleph_1$ (Theorem A). In fact, we present two different arguments for this. The first (Theorem~\ref{b(EDfin)=aleph1}), discovered by the first author in 2009, is based on eventually different functions and was used originally to show $\bfrak(\mathcal{ED}_\textrm{fin})=\aleph_1$ (Example \ref{ExpFragId}(3)); in fact it can be used for the ideals in Example \ref{ExpFragId}(4) as well. The second method (Theorems~\ref{b(unifsubmeas)=aleph1} and~\ref{b(meas)=aleph1}), based on independent functions, seems to apply to a larger class of ideals (including those of Examples~\ref{ExpFragId}(3) and (4)). Still, we decided to include the first argument since it may be useful in other contexts.

\begin{theorem}\label{b(EDfin)=aleph1}
   $\bfrak(\mathcal{ED}_\textrm{fin})=\aleph_1$.
\end{theorem}

\begin{proof}
   Consider $\langle a_i,\varphi_i\rangle_{i<\omega}$ as in Example \ref{ExpFragId}(3). Construct a disjoint family $\Awf=\{A_n\ :\ n<\omega\}$ of subsets of $\omega$ such that, for each $n<\omega$, $\lim_{i\to+\infty}|A_n\cap a_i|=+\infty$. To see that this can be done, construct, by induction on $n<\omega$, a $\leq$-increasing sequence $\langle e_n\rangle_{n<\omega}$ of functions in $\omega^\omega$ such that $e_n\leq id_\omega$, $e_n$ converges to infinity and $e_{n+1}-e_n$ converges to infinity. For each $i<\omega$, consider a bijection $g_i:i+1\to a_i$ and put $A_n:=\bigcup_{n<\omega}g_i[[e_n(i),e_{n+1}(i))]$ and $\bar{A}_n:=\bigcup_{k\leq n}A_k$.

   For each $n<\omega$ let $N_n$ be such that $A_n\cap a_i\neq\varnothing$ for every $i\geq N_n$. As $\lim_{i\to+\infty}|A_n\cap a_i|=+\infty$, there exists a pairwise eventually different family of functions $\{f_{n,\alpha}\}_{\alpha<\omega_1}$ in $\prod_{i\geq N_n}(A_n\cap a_i)$, that is, if $\alpha\neq\beta$ then $\forall^\infty_i(f_{n,\alpha}(i)\neq f_{n,\beta}(i))$.

   Construct, by induction, a $\subseteq_{\mathcal{ED}_\textrm{fin}}$-increasing sequence $\Bwf=\{B_\alpha\}_{\alpha<\omega_1}$ that is $\mathcal{ED}_\textrm{fin}$-orthogonal with $\Awf$ and such that $\forall_{\beta<\alpha}\forall^\infty_{n<\omega}(\ran f_{n,\beta}\subseteq B_\alpha)$. Indeed, let $B_0=\varnothing$ and $B_{\alpha+1}=B_\alpha\cup\bigcup_{n<\omega}\ran f_{n,\alpha}$. For the limit step, if $\alpha<\omega_1$ limit, let $B_\alpha=\bigcup_{n<\omega}\big[(B_{\alpha_n}\menos\bar{A}_n)\cup\bigcup_{k<n}\ran f_{n,\beta_k}\big]$ where $\{\alpha_n\}_{n<\omega}$ is a strictly increasing sequence converging to $\alpha$ and $\alpha=\{\beta_k\ /\ k<\omega\}$ is an enumeration. Note that $B_{\alpha_n}\menos B_\alpha\subseteq B_{\alpha_n}\cap\bar{A}_n\in\mathcal{ED}_\textrm{fin}$, $B_\alpha\cap A_n\subseteq\bigcup_{k<n}\big( (B_{\alpha_k} \cap {A}_n) \cup\ran f_{n,\beta_k}\big)\in\mathcal{ED}_\textrm{fin}$ and $\forall_{n>k}(\ran f_{n,\beta_k}\subseteq B_{\alpha})$.

   We claim that $\langle\Awf,\Bwf\rangle$ is an $\mathcal{ED}_\textrm{fin}$-gap. Assume the contrary, so there exists a $C$ that separates $\langle\Awf,\Bwf\rangle$. By recursion on $n<\omega$, construct a decreasing chain $\{X_n\}_{n<\omega}$ of infinite subsets of $\omega$ and $F_n\subseteq\omega_1$ finite such that $\forall_{\alpha\in\omega_1\menos F_n}\forall^\infty_{i\in X_n}(f_{n,\alpha}(i)\notin C)$. Start with $X_{-1}=\omega$. Suppose that $X_n$ has been constructed ($n\geq-1$). As $C\cap A_{n+1}\in\mathcal{ED}_\textrm{fin}$, there exists an $l<\omega$ such that $\forall_{i<\omega}(|C\cap A_{n+1}\cap a_i|\leq l)$. By recursion in $j<\omega$ choose, if possible, $Y_j\subseteq\omega$ infinite and $\alpha_j\in\omega_1\menos\{\alpha_k\ /\ k<j\}$ such that $Y_0=X_n$, $Y_{j+1}\subseteq Y_j$ and $\forall_{i\in Y_{j+1}}(f_{n+1,\alpha_j}(i)\in C)$. Note that this construction must stop at $l$, at the latest, that is, $Y_{l+1}$ and $\alpha_l$ cannot exist. For otherwise, as $\{f_{n+1,\alpha}\}_{\alpha<\omega_1}$ is a sequence of pairwise eventually different functions, there exists an $i\in Y_{l+1}$ such that all $f_{n+1,\alpha_j}(i)$ are different for $j\leq l$ and then, as $\{f_{n+1,\alpha_j}(i)\ /\ j\leq l\}\subseteq C\cap A_{n+1}\cap a_i$, $|C\cap A_{n+1}\cap a_i|>l$, which is impossible. Now, once the construction stops at $l_0\leq l$, $F_{n+1}:=\{\alpha_j\ /\ j<l_0\}$ and $X_{n+1}:=Y_{l_0}$ are as required.

   Let $X$ be a pseudo-intersection of $\{X_n\}_{n<\omega}$, that is, $X\subseteq\omega$ is infinite and $X\subseteq^* X_n$ for all $n<\omega$. Choose $\alpha<\omega_1$ strictly above all the ordinals in $\bigcup_{n<\omega}F_n$. Note that, for any $n<\omega$, $\ran f_{n,\alpha}\subseteq B_{\alpha+1}$ and $\forall^\infty_{i\in X}(f_{n,\alpha}(i)\notin C)$. On the other hand, as $B_{\alpha+1}\subseteq_{\mathcal{ED}_\textrm{fin}}C$, there exists a $k<\omega$ such that $\forall_{i<\omega}(|a_i\cap B_{\alpha+1}\menos C|\leq k)$. We can find an $i\in X$ such that $i\geq N_n$ and $f_{n,\alpha}(i)\notin C$ for all $n\leq k$. Then, $\{f_{n,\alpha}(i)\ /\ n\leq k\}\subseteq a_i\cap B_{\alpha+1}\menos C$ so, as $f_{n,\alpha}(i)\in A_n$, it is clear that $|a_i\cap B_{\alpha+1}\menos C|>k$, a contradiction.
\end{proof}

\begin{corollary}\label{b(LinGr)=aleph1}
   If $\Iwf$ is a non-trivial ideal as defined in Example \ref{ExpFragId}(4), then $\bfrak(\Iwf)=\aleph_1$.
\end{corollary}

\begin{proof}
   By Remark \ref{RemGaps}(2), we may assume that $|a_i|=(i+1)g(i)$ for any $i<\omega$. Let $\{a_{i,j}\}_{j<i+1}$ be a partition of $a_i$ into sets of size $g(i)$. Let $\{b_i\}_{i<\omega}$ be the interval partition of $\omega$ such that $|b_i|=i+1$ and let $b_i=\{k_{i,j}\ /\ j<i+1\}$ be an enumeration. Define the finite-one function $h:\omega\to\omega$ such that $h^{-1}[\{k_{i,j}\}]=a_{i,j}$. Note that, for $x\subseteq\omega$, $x\in\mathcal{ED}_\textrm{fin}$ iff $h^{-1}[x]\in\Iwf$, so $F:\Pwf(\omega)/\mathcal{ED}_\textrm{fin}\to\Pwf(\omega)/\Iwf$, $F([x])=[h^{-1}[x]]$, is an embedding (of Boolean algebras).

   It suffices to show that $F$ preserves gaps. Let $\langle\Awf,\Bwf\rangle$ be $\mathcal{ED}_\textrm{fin}$-orthogonal. If the $\Iwf$-orthogonal pair $\langle\{h^{-1}[A]\ /\ A\in\Awf\},\{h^{-1}[B]\ /\ B\in\Awf\}\rangle$ is separated by a subset $C$ of $\omega$, then $H(C):=\bigcup_{i<\omega}\{k_{i,j}\ /\ |C\cap a_{i,j}|\geq\frac{1}{2}g(i)\}$ separates $\langle\Awf,\Bwf\rangle$. Indeed, if $A\in\Awf$ then there exists an $l<\omega$ such that $|C\cap h^{-1}[A\cap b_i]|=|C\cap h^{-1}[A]\cap a_i|\leq l\cdot g(i)$ for all $i<\omega$. Therefore, $|H(C)\cap A\cap b_i|\leq 2\cdot l$ for every $i<\omega$, so $H(C)\cap A\in\mathcal{ED}_\textrm{fin}$. Likewise, as $\omega\menos H(C)=\bigcup_{i<\omega}\{k_{i,j}\ /\ |a_{i,j}\menos C|>\frac{1}{2}g(i)\}$, $B\menos H(C)\in\mathcal{ED}_\textrm{fin}$ for any $B\in\Bwf$.
\end{proof}

We will obtain a generalization of the previous two results. Before, we introduce the following characterization of fragmented not gradually fragmented ideals.

\begin{lemma}\label{notgradchar}
   Let $\Iwf=\Iwf\langle a_j,\varphi_j\rangle_{j<\omega}$ be a fragmented not gradually fragmented ideal. Then, there exist $k<\omega$, a sequence $\langle C_i\rangle_{i<\omega}$ of pairwise disjoint infinite subsets of $\omega$ and a sequence $\{l_i\}_{i<\omega}$ of natural numbers such that, for any $i<\omega$ and $j\in C_i$, there exists a pairwise disjoint family $B_j$ of subsets of $a_j$ such that $|B_j|=l_i$, $\forall_{b\in B_j}(0<\varphi_j(b)\leq k)$ and $i<\varphi_j(\bigcup B_j)\leq i+k$.
\end{lemma}

\begin{proof}
   As $\Iwf$ is not gradually fragmented, there exists a $k<\omega$ such that, for any $i<\omega$, there is an $l'_i<\omega$ and $W'_i\subseteq\omega$ infinite such that, for any $j\in W'_i$, there exists a $B'_j\subseteq\Pwf(a_j)$ of size $\leq l'_i$ such that $\forall_{b\in B'_j}(\varphi_j(b)\leq k)$ and $\varphi_j(\bigcup B'_j)>i$. By taking complements between the members of $B'_j$, it is easy to find a pairwise disjoint family $B''_j\subseteq\Pwf(a_j)$ of size $\leq l'_i$ such that $\bigcup B''_j=\bigcup B'_j$ and $\forall_{b\in B'_j}(0<\varphi_j(b)\leq k)$. A similar argument as in the proof of (ii) implies (iii) of Lemma \ref{tallfrag} shows that there is a $B_j\subseteq B''_j$ such that $i<\varphi_j(\bigcup B_j)\leq i+k$.

   It is clear that, for each $i<\omega$, there exists an $l_i\leq l'_i$ and $W''_i\subseteq W'_i$ infinite such that $|B_j|=l_i$ for all $j\in W''_i$. Finally find $C_i \subseteq W''_i$ infinite and pairwise disjoint.
\end{proof}

The first part for our generalization focuses on the class of fragmented not gradually fragmented ideals that can be characterized by uniform submeasures. For a finite set $a$, we say that a submeasure $\varphi:\Pwf(a)\to[0,+\infty)$ is \emph{uniform} if it only depends on the size of the sets.

\begin{theorem}\label{b(unifsubmeas)=aleph1}
   Let $\Iwf=\Iwf\langle a_j,\varphi_j\rangle_{j<\omega}$ be a fragmented not gradually fragmented ideal such that all the $\varphi_j$ are uniform submeasures. Then, $\bfrak(\Iwf)=\aleph_1$.
\end{theorem}

\begin{proof}
   Let $k$ be as in Lemma \ref{notgradchar}. By multiplying all the submeasures by $\frac{1}{k}$ (see Remark \ref{RemFragId}), we may assume that $k=1$. Therefore, we can write $\Iwf=\Iwf\langle a_{i,j,k},\varphi_{i,j,k}\rangle_{i,j,k<\omega}$, where the submeasures are uniform, such that there is a sequence $\{l_i\}_{i<\omega}$ of natural numbers such that, for any $i<\omega$, there is $W_i\subseteq\omega\times\omega$ infinite and, for any $(j,k)\in W_i$, there exists $B_{i,j,k}$ a pairwise disjoint family of subsets of $a_{i,j,k}$ such that $|B_{i,j,k}|=l_i$, $\forall_{b\in B_{i,j,k}}(0<\varphi_{i,j,k}(b)\leq 1)$ and $i<\varphi_{i,j,k}(\bigcup B_{i,j,k})\leq i+1$. By Remark \ref{RemGaps}(2), we may assume that $W_i=\omega\times\omega$ and $a_{i,j,k}=\bigcup B_{i,j,k}$ for all $i,j,k<\omega$. Also, without loss of generality, $\varphi_{i,j,k}=\lceil\varphi_{i,j,k}\rceil$, so $\forall_{b\in B_{i,j,k}}(\varphi_{i,j,k}(b)=1)$ and $\varphi_{i,j,k}(a_{i,j,k})=i+1$.

   Fix $i,j,k<\omega$. For each $m\leq i+1$ let $s_{i,j,k}(m)$ be the maximal $n\leq|a_{i,j,k}|$ such that all the subsets of $a_{i,j,k}$ of size $n$ have submeasure equal to $m$. By uniformity, it is clear that $s_{i,j,k}(m)$ exists and $s_{i,j,k}(m)<s_{i,j,k}(m+1)$ for $m\leq i$. Also, note that $s_{i,j,k}(0)=0$. By induction on $m\leq i$, it is easy to prove that $m\cdot s_{i,j,k}(1)\leq s_{i,j,k}(m)$. We may also assume that $|a_{i,j,k}|=s_{i,j,k}(i)+1$.

   For each $0<m<\omega$, $i\geq m$ and $j,k<\omega$, let $n_{i,j,k}(m)$ and $r_{i,j,k}(m) < s_{i,j,k}(m)$ be such  that $|a_{i,j,k}|=s_{i,j,k}(m)\cdot n_{i,j,k}(m)+r_{i,j,k}(m)$. Note that $m\cdot s_{i,j,k}(1)\cdot n_{i,j,k}(m)\leq s_{i,j,k}(m)\cdot n_{i,j,k}(m)\leq|a_{i,j,k}|\leq l_i\cdot s_{i,j,k}(1)$, where the last inequality holds because $|B_{i,j,k}|=l_i$. Therefore, $n_{i,j,k}(m)\leq l_i$. Thus, for a fixed $i<\omega$, there is an infinite $V_i\subseteq\omega\times\omega$ such that, for all $0<m\leq i$, there is an $n_i(m)\leq l_i$ such that $n_{i,j,k}(m)=n_i(m)$ for all $(j,k)\in V_i$. Again, by Remark \ref{RemGaps}(2), we may assume that $V_i=\omega\times\omega$.

   Fix $k<\omega$ and put
   \[P_{k}=\prod\big\{\{x\subseteq a_{i,j,k}\ /\ |x|=s_{i,j,k}(k)\}\ /\ j<\omega, i\geq k\}\big\}.\]
   Say that a family $\Fwf\subseteq P_{k}$ is \emph{independent} if, for any finite $F\subseteq\Fwf$ and for all $i\geq k$, there are infinitely many $j$'s such that $\{f(i,j)\ /\ f\in F\}$ is either pairwise disjoint or its union is $a_{i,j,k}$. It is easy to see that adding a Cohen real adds a real $c\in P_{k}$ such that, whenever $\Fwf$ is an independent family in the ground model, $\Fwf\cup\{c\}$ is independent in the extension. Therefore, there exists an independent family $\Fwf_{k}\subseteq P_{k}$ of size $\aleph_1$. Say $\Fwf_k =\{f_{k,\alpha}\ /\ \alpha<\omega_1\}$.

   For $k<\omega$ let $A_k:=\bigcup_{i,j<\omega}a_{i,j,k}$ and for $\alpha<\omega_1$ let $B_\alpha :=\bigcup\left\{ f_{k,\alpha}(i,j)\ :\ j,k<\omega,\right.$  $\left.i\geq k\right\}$. As $A_k\cap B_\alpha=\bigcup_{j<\omega, i\geq k}f_{k,\alpha}(i,j)$ and $\varphi_{i,j,k} ( f_{k,\alpha} ( i, j ) ) = k$, we get that $A_k\cap B_\alpha\in\Iwf$, that is, $\langle\{A_k\}_{k<\omega},\{B_\alpha\}_{\alpha<\omega_1}\rangle$ is $\Iwf$-orthogonal. We want to show that it is an $\Iwf$-gap.

   Assume that $B$ separates $\langle\{A_k\}_{k<\omega},\{B_\alpha\}_{\alpha<\omega_1}\rangle$. Find $\Gamma\subseteq\omega_1$ uncountable and $m<\omega$ such that, for all $\alpha\in\Gamma$, $\bar{\varphi}(B_\alpha\menos B)\leq m$ (here, $\bar{\varphi}(X)=\sup_{i,j,k<\omega}\{\varphi_{i,j,k}(X\cap a_{i,j,k})\}$). Let $k>2m$ and find $i\geq k$ such that $i+1-k>2\cdot\bar{\varphi}(A_k\cap B)$. Choose $H\subseteq\Gamma$ of size $n_i(k)$. By independence, there are infinitely many $j$'s such that $\{f_{k,\alpha}(i,j)\ /\ \alpha\in H\}$ is a disjoint family because, in the case that its union is $a_{i,j,k}$, we have $r_{i,j,k}(k)=0$ and the family will be disjoint anyway. Work with one of these $j$'s. For any $\alpha\in H$, as $\varphi_{i,j,k}(f_{k,\alpha}(i,j)\menos B)\leq m$ and $\varphi_{i,j,k}(f_{k,\alpha}(i,j))=k$, we obtain $|f_{k,\alpha}(i,j)\menos B|<\frac{1}{2}|f_{k,\alpha}(i,j)|=\frac{1}{2}s_{i,j,k}(k)$. Thus, $|\bigcup_{\alpha\in H}f_{k,\alpha}(i,j)\menos B|<\frac{1}{2}n_i(k)s_{i,j,k}(k)=\frac{1}{2}|\bigcup_{\alpha\in H}f_{k,\alpha}(i,j)|$, which implies that
   \[\varphi_{i,j,k}(\bigcup_{\alpha\in H}f_{k,\alpha}(i,j)\cap B)\geq\frac{1}{2}\varphi_{i,j,k}(\bigcup_{\alpha\in H}f_{k,\alpha}(i,j)).\]
   But, because $r_{i,j,k}(k) < s_{i,j,k}(k)$, we have $\varphi_{i,j,k}(\bigcup_{\alpha\in H}f_{k,\alpha}(i,j))\geq i+1-k>2\cdot\varphi_{i,j,k}(a_{i,j,k}\cap B)$, a contradiction.
\end{proof}

\begin{corollary}
   Let $\Iwf$ be a fragmented not gradually fragmented ideal as in Example \ref{ExpFragId}(5). Then, $\bfrak(\Iwf)=\aleph_1$.
\end{corollary}

The second part corresponds to fragmented not gradually fragmented ideals that can be characterized by measures.

\begin{theorem}\label{b(meas)=aleph1}
   Let $\Iwf=\Iwf\langle a_j,\varphi_j\rangle_{j<\omega}$ be a somewhere tall fragmented ideal such that all the $\varphi_j$ are measures. Then, $\Iwf$ is not gradually fragmented and $\bfrak(\Iwf)=\aleph_1$.
\end{theorem}

\begin{proof}
   Without loss of generality, by restricting the ideal to an $\Iwf$-positive set, we may assume that $\Iwf$ is tall and that, for all $j<\omega$ and $k\in a_j$, $0<\varphi_j(\{k\})\leq 1$. Let $i<\omega$ and $L_i:=\{j<\omega\ /\ \varphi_j(a_j)\geq i+1\}$, which is infinite. To see that $\Iwf$ is not gradually fragmented we show that, for each $j\in L_i$, there exists a pairwise disjoint family $B_j\subseteq\Pwf(a_j)$ of size $\leq2\cdot i$ such that
   $\forall_{b\in B_j}(\varphi_j(b)\leq 1)$ and $\varphi_j(\bigcup B_j)>i$.

   \begin{claim}
      Let $x\subseteq a_j$ such that $\varphi_j(x)\geq 1$. Then, there exists a $y\subseteq x$ such that $\frac{1}{2}<\varphi_j(y)\leq 1$.
   \end{claim}

   \begin{proof}
      Let $y$ be a subset of $x$ of maximal measure $\leq 1$. If $\varphi_j(y)\leq\frac{1}{2}$ there exists a $k\in x\menos y$ so, as $\varphi_j(y\cup\{k\})>1$, we get that $\frac{1}{2}<\varphi(\{k\})\leq 1$, which contradicts the maximality of $y$.
   \end{proof}

   Construct $B_j=\{b_{j,k}\ /\ k<l\}$ by recursion on $k$, where $\frac{1}{2}<\varphi_j(b_{j,k})\leq 1$ (the $l$ is defined at the end). Assume we have got $b_{k'}$ for $k'<k$. If $\varphi_j(\bigcup_{k'<k} b_{j,k'})>i$ put $l=k$ and stop the recursion. Otherwise, $\varphi_j(a_j\menos\bigcup_{k'<k} b_{j,k'})\geq1$, so we get $b_{j,k}\subseteq a_j\menos\bigcup_{k'<k} b_{j,k'}$ by application of the claim. If this recursion reaches $2\cdot i$ steps, put $l=2\cdot i$. Note that $\varphi_j(\bigcup B_j)=\sum_{k<2\cdot i}\varphi_j(b_{j,k})>i$.

   Like in the first part of the proof of Theorem \ref{b(unifsubmeas)=aleph1}, we may assume that $\Iwf=\Iwf\langle a_{i,j,k},$ $\varphi_{i,j,k}\rangle_{i,j,k<\omega}$ is given by measures and that there is a sequence $\{l_i\}_{i<\omega}$ of natural numbers such that, for any $i,j,k<\omega$, there exists $B_{i,j,k}$ a pairwise disjoint family of subsets of $a_{i,j,k}$ such that $|B_{i,j,k}|=l_i$, $\forall_{b\in B_{i,j,k}}(0<\varphi_{i,j,k}(b)\leq 1)$, $\bigcup B_{i,j,k}=a_{i,j,k}$ and $i<\varphi_{i,j,k}(a_{i,j,k})\leq i+1$.

   For $i,j,k<\omega$, there exist $n_{i.j,k}<\omega$ and $r_{i,j,k}\leq k$ such that $i+1=(k+1)\cdot n_{i,j,k}+r_{i,j,k}$. As $(k+1)\cdot n_{i,j,k}\leq i+1$, we may assume that there is an $n_{i,k}<\omega$ such that $n_{i,j,k}=n_{i,k}$ for all but finitely many $j<\omega$. To see this, construct a decreasing family $\{W_{i,k}\}_{i,k<\omega}$ (with respect to a well order of $\omega\times\omega$) of infinite subsets of $\omega$ such that, for each $i,k<\omega$, there is an $n_{i,k}<\omega$ such that $n_{i,j,k}=n_{i,k}$ for all $j\in W_{i,k}$. Let $W$ be a pseudo-intersection of $\{W_{i,k}\}_{i,k<\omega}$. By restricting the ideal, we may assume that $W=\omega$ (the set corresponding to the $j$ coordinates). Also note that, for fixed $k$, the sequence $\{n_{i,k}\}_{i<\omega}$ converges to infinity because $i+1<(k+1)\cdot(n_{i,k}+1)$.

   Start as in the proof of Theorem \ref{b(unifsubmeas)=aleph1} but change ``$|x|=s_{i,j,k}(k)$'' to ``$\varphi_{i,j,k}(x)\in(k,k+1]$'' in the definition of $P_{k}$. After choosing $\Gamma$ and $m$, proceed as follows. Choose $k>m$ and find $i\geq k$ such that $\bar{\varphi}(A_k\cap B)<n_{i,k}-1$. Now, for $H\subseteq\Gamma$ of size $n_{i,k}-1$, by independence there are infinitely many $j$'s such that $n_{i,j,k}=n_{i,k}$ and $\{f_{k,\alpha}(i,j)\ /\ \alpha\in H\}$ is a disjoint family. Work with one of these $j$'s. As $\varphi_{i,j,k}(\bigcup_{\alpha\in H}f_{k,\alpha}(i,j)\menos B)\leq m\cdot(n_{i,k}-1)$, $\varphi_{i,j,k}(\bigcup_{\alpha\in H}f_{k,\alpha}(i,j)\cap B)>(n_{i,k}-1)\cdot(k-m)\geq n_{i,k}-1$. Thus, $\varphi_{i,j,k}(B\cap a_{i,j,k})>n_{i,k}-1$, a contradiction.
\end{proof}

Note that if $\Iwf$ is a fragmented not gradually fragmented ideal and $\omega=X\cup Y$ is a disjoint union, then $\Iwf\frestr X$ or $\Iwf\frestr Y$ is not gradually fragmented. Because of this, we can mix Theorems \ref{b(unifsubmeas)=aleph1} and \ref{b(meas)=aleph1} to obtain

\begin{corollary}
   Let $\Iwf=\Iwf\langle a_j,\varphi_j\rangle_{j<\omega}$ be a fragmented not gradually fragmented ideal such that, for all but finitely many $j<\omega$, either $\varphi_j$ is a measure or a uniform submeasure. Then, $\bfrak(\Iwf)=\aleph_1$.
\end{corollary}

To finish the section, we explain a way of how to obtain a fragmented not gradually fragmented ideal from a fragmented ideal. Let $\Iwf=\Iwf\langle a_i,\varphi_i\rangle_{i<\omega}$ be a fragmented ideal. Now, let $\langle a_{i,j}\rangle_{i,j<\omega}$ be a partition of $\omega$ such that, for a fixed $i<\omega$ and all $j<\omega$, $|a_{i,j}|=|a_i|$ and $\varphi_{i,j}:\Pwf(a_{i,j})\to[0,+\infty)$ is the \emph{submeasure associated with $\langle a_i,\varphi_i\rangle$}, that is, if $h_{i,j}:a_{i,j}\to a_i$ is the (unique) strictly increasing bijection, then $\varphi_{i,j}(x)=\varphi_i(h_{i,j}[x])$ for any $x\subseteq a_{i,j}$. Let $\hat{\Iwf}$ be the fragmented ideal associated to $\langle a_{i,j},\varphi_{i,j}\rangle_{i,j<\omega}$. Roughly speaking, $\hat{\Iwf}$ is the ideal obtained by taking countably many copies of the ideal $\Iwf$.

\begin{lemma}\label{DestrGrad}
   With the notation of the previous paragraph,
   \begin{enumerate}[(a)]
      \item if $\Iwf$ is nowhere tall, then $\hat{\Iwf}$ is;
      \item if $\Iwf$ is somewhere tall, then $\hat{\Iwf}$ is not gradually fragmented.
   \end{enumerate}
\end{lemma}

\begin{proof}
   \begin{enumerate}[(a)]
      \item Let $X\subseteq\omega$ be $\hat{\Iwf}$-positive, that is, $\{\varphi_{i,j}(X\cap a_{i,j})\}_{i,j<\omega}$ is an unbounded set of non-negative reals. Then, there exist $W\subseteq\omega$ infinite and a function $g:W\to\omega$ such that $\{\varphi_{i,g(i)}(X\cap a_{i,g(i)})\}_{i\in W}$ converges to infinity. Put $X_1=\bigcup_{i\in W}X\cap a_{i,g(i)}\subseteq X$ and $X'_1=\bigcup_{i\in W}(h_{i,g(i)}[X\cap a_{i,g(i)}])$. Clearly, $\langle X_1,\hat{\Iwf}\frestr X_1\rangle$ and $\langle X'_1,\Iwf\frestr X'_1\rangle$ are isomorphic and, as the second ideal is not tall, neither is the first ideal.
      \item Without loss of generality, we may assume that $\Iwf$ is tall and non-trivial. By Lemma \ref{tallfrag}, let $k$ be such that $\varphi_i(\{c\})\leq k$ for all $c\in a_i$ and $i<\omega$. Now let $m<\omega$ be arbitrary. Choose an $i<\omega$ such that $\varphi_i(a_i)>m$ and let $l:=|a_i|$. Note that, for any $j<\omega$, the family $B_{i,j}=\{\{c\}\ /\ c\in a_{i,j}\}$ has size $l$ and satisfies $\forall_{b\in B_{i,j}}(\varphi_{i,j}(b)\leq k)$ and $\varphi_{i,j}(\bigcup B_{i,j})>m$.
   \end{enumerate}
\end{proof}

\begin{corollary}\label{b(DestrGradUnif)=aleph1}
    Let $\Iwf=\Iwf\langle a_i,\varphi_i\rangle_{i<\omega}$ be a somewhere tall fragmented ideal such that, for any $i<\omega$, either $\varphi_i$ is uniform or is a measure. Then, $\bfrak(\hat{\Iwf})=\aleph_1$.
\end{corollary}

%
%

\section{Destroying gaps of gradually fragmented ideals}\label{SecDestroyGradFragGaps}

We present in this section a way to destroy Rothberger gaps for gradually fragmented ideals by a ccc poset. Moreover, for the case of an ideal like in Example \ref{ExpFragId}(2), we can even find a natural cardinal invariant that is less than or equal to its Rothberger number. As a consequence of our discussion we obtain two basic $\thzfc$-results: the Rothberger number of any gradually fragmented ideal is larger or equal to $\add(\Nwf)$ (Corollary~\ref{b(gradfrag)aboveAddN}) and the Rothberger number of any nowhere tall fragmented ideal is $\bfrak$ (Corollary~\ref{b(notpseudotall)=b}).

To fix some notation, for $b,h\in\omega^\omega$ let $\R_b:=\prod_{i<\omega}b(i)$ and $S(b,h) := \prod_{i<\omega}[b(i)]^{\leq h(i)}$. Also,
for $n<\omega$ put $S_n(b,h) := \prod_{i<n}[b(i)]^{\leq h(i)}$ and $S_{<\omega}(b,h)$ $:= \bigcup_{n<\omega}S_n(b,h)$. The forcing notions and cardinal invariants involved in the destruction of Rothberger gaps of gradually fragmented ideals are, respectively, parameterized versions of the localization forcings and of the cardinal invariant $\add(\Nwf)$.

\begin{definition}[Localization posets and cardinal invariants]\label{DefLocPoset}
   Let $b\in\omega^\omega$ such that $b>0$ and let $h\in\omega^\omega$ be a non-decreasing function.
   \begin{enumerate}[(1)]
      \item Define $\blocfrak(b,h)$ as the minimal size of a subset of $\R_b$ that cannot be $\in^*$-bounded by any slalom in $S(b,h)$ (if it exists).
      \item For $\Fwf\subseteq\R_b$, define the poset
            \[\Qor^h_{b,\Fwf}:=\{(s,F)\ /\ s\in S_{<\omega}(b,h)\textrm{,\ }F\subseteq\Fwf\textrm{\ and\ }|F|\leq h(|s|)\}\]
            ordered by $(s',F')\leq(s,F)$ iff $s\subseteq s'$, $F\subseteq F'$ and $\forall_{i\in[|s|,|s'|)}(\{x(i)\ /\ x\in F\}\subseteq s'(i))$. Put
            $\Qor^h_b:=\Qor^h_{b,\R_b}$.
   \end{enumerate}
\end{definition}

\begin{lemma}\label{LocPosetLemma}
   In the notation of Definition \ref{DefLocPoset}, if $h$ converges to infinity, then $\Qor^h_{b,{\Fwf}}$ is $\sigma$-linked and generically adds a slalom in $S(b,h)$ that $\in^*$-dominates all the reals in $\Fwf$. In particular, $\Qor^h_{b}$ generically adds a slalom in $S(b,h)$ that $\in^*$-dominates all the ground model reals in $\R_b$.
\end{lemma}

\begin{proof}
   $\sigma$-linked is witnessed by $Q_s:=\{(t,F)\in\Qor^h_{b,\Fwf}\ /\ t=s\textrm{\ and\ }|F|\leq h(|s|)/2\}$. Convergence of $h$ to infinity implies the density of $\bigcup \{ Q_s \ /\ s\in S_{<\omega}(b,h) \}$. If $\dot{G}$ is the $\Qor^h_{b,\Fwf}$-name of the generic subset, then
   $\bigcup\dom\dot{G}$ is the name of the slalom that $\in^*$-dominates all the reals in $\Fwf$.
\end{proof}

The Bartoszy\'nski characterization of $\add(\Nwf)$ (Theorem \ref{BartcharAdd(N)}) implies that $\add(\Nwf)\leq\blocfrak(b,h)$ when $h$ converges to infinity.

\begin{theorem}\label{DestroyGradFragGaps}
   Let $\Iwf$ be a gradually fragmented ideal. Then, there exists a function $b\in\omega^\omega$ and a $\Dor$-name $\dot{h}$ of a non-decreasing function in $\omega^\omega$ that converges to infinity such that the forcing $\Dor\ast\Qor^{\dot{h}}_b$ destroys all the $\Iwf$-Rothberger gaps in the ground model.
\end{theorem}

\begin{proof}
   In $V$ (the ground model), let $\Iwf=\Iwf\langle a_i,\varphi_i\rangle_{i<\omega}$ be a gradually fragmented ideal and $f\in\omega^\omega$ a function that witnesses its gradual fragmentation. Let $d\in\omega^\omega$ be a strictly increasing dominating real added generically over $V$ by $\Dor$. In $V[d]$, by the gradual fragmentation of $\Iwf$, find a non-decreasing sequence $\{N_l\}_{l<\omega}$ of natural numbers that converges to infinity and such that
   \begin{multline*}
     (+)\ \forall_{n\leq l}\forall_{i\geq N_l}\forall_{B\subseteq\Pwf(a_i)}\big[\big(|B|\leq l\textrm{\ and\ }\forall_{x\in B}(\varphi_i(x)\leq d(n))\big)\\
     \imp\varphi_i(\bigcup B)\leq f(d(n))\big].
   \end{multline*}
   $N_0=0$ can be assumed. Let $h:\omega\to\omega$ be defined as $h(i)=l$ when $i\in[N_l,N_{l+1})$.

   Back in $V$, let $\dot{h}$ be the $\Dor$-name of $h$ and put $b(i)=\Pwf(a_i)$. Now, let $\langle\Awf,\Bwf\rangle$ be $\Iwf$-orthogonal with $|\Awf|=\aleph_0$ and we show that $\Dor\ast\Qor^{\dot{h}}_b$ adds a subset of $\omega$ that separates $\langle\Awf,\Bwf\rangle$, moreover, we can even find a $\Dor$-name $\dot{\Fwf}$ of a subset of $\R_b$ of size $\leq|\Bwf|$ such that $\Dor\ast\Qor^{\dot{h}}_{b,\dot{\Fwf}}$ adds such a subset of $\omega$. Put $\Awf=\{A_n\ /\ n<\omega\}$. Without loss of generality, we may assume that $\Awf$ is a partition of $\omega$ and that, for each $i<\omega$,  $\forall_{n>i}(A_n\cap a_i=\varnothing)$, so $\{A_n\cap a_i\}_{n\leq i}$ becomes a partition of $a_i$. For each $B\in\Bwf$, let $g_B\in\omega^\omega$ be such that $g_B(n)=\lceil\bar{\varphi}(B\cap A_n)\rceil$. Step into $V[d]$ and, for each $B\in\Bwf$, define $x_B\in\R_b$ such that $x_B(i)=\bigcup_{n\leq i}x_B(i,n)$ where $x_B(i,n)=B\cap A_n\cap a_i$ if $\varphi_i(B\cap A_n\cap a_i)\leq d(n)$ and, otherwise, $x_B(i,n)=\varnothing$. Put $\Fwf:=\{x_B\ /\ B\in\Bwf\}$.

   Let $\psi$ be a slalom in $S(b,h)$ added generically over $V[d]$ by $\Qor^h_{b,\Fwf}$. Now, work in $V[d][\psi]$. Without loss of generality, we may assume that, for every $i<\omega$,
   $\forall_{x\in\psi(i)}\forall_{n\leq i}(\varphi_i(x\cap A_n)\leq d(n))$ (just take out those $x$ of $\psi(i)$ that do not satisfy that property). Put $C:=\bigcup_{i<\omega}\bigcup_{x\in\psi(i)}x$. This $C$ separates $\langle\Awf,\Bwf\rangle$.
   \begin{itemize}
     \item \emph{$C\cap A_n\in\Iwf$ for all $n<\omega$, moreover, $\varphi_i(C\cap A_n\cap a_i)\leq f(d(n))$ for all $i\geq N_n$.} Let $l<\omega$ be such that $i\in[N_l,N_{l+1})$, so $|\psi(i)|\leq h(i)=l$ and, by (+), as $n\leq l$, $C\cap A_n\cap a_i=\bigcup_{x\in\psi(i)}(x\cap A_n)$ and $\varphi_i(x\cap A_n)\leq d(n)$ for all $x\in\psi(i)$, we have that $\varphi_i(C\cap A_n\cap a_i)\leq f(d(n))$.
     \item \emph{$B\menos C\in\Iwf$ for all $B\in\Bwf$.} Note that $g_B\leq^* d$, so there exists an $m<\omega$ such that, for every $n\geq m$ and $i \geq n$, $\varphi_i(B\cap A_n\cap a_i)\leq d(n)$. Also, as $x_B\in^*\psi$, we may assume (by enlarging $m$) that $x_B(i)\in\psi(i)$ for all $i\geq m$. Then, $B\cap A_n\cap a_i\subseteq C\cap A_n\cap a_i$ for all $i \geq n\geq m$, so
         $B\cap(\bigcup_{n\geq m}A_n)\subseteq C\cap(\bigcup_{n\geq m}A_n)$. As $A_n\cap B\in\Iwf$ for any $n<\omega$, it follows that $B\subseteq_\Iwf C$.
   \end{itemize}
\end{proof}

The previous proof also indicates that the forcing $\Dor\ast\Loc^{\dot{h}}$ destroys the Rothberger gaps of the ground model for any gradually fragmented ideal $\Iwf$. But, as any localization forcing $\Loc^{h'}$ adds a dominating real, the following result comes as a consequence.

\begin{corollary}\label{b(gradfrag)aboveAddN}
   If $\Iwf$ is a gradually fragmented ideal, then $\add(\Nwf)\leq\bfrak(\Iwf)$.
\end{corollary}

\begin{proof}
   Let $\Iwf=\Iwf\langle a_i,\varphi_i\rangle_{i<\omega}$ and $\langle\Awf,\Bwf\rangle$ be $\Iwf$-orthogonal where $\Awf$ is a partition of $\omega$ and $|\Bwf|<\add(\Nwf)$. Let $\chi$ be a large enough regular cardinal and $M\preceq H_\chi$ such that $\Bwf\cup\{\Awf,\Bwf,\langle a_i,\varphi_i\rangle_{i<\omega}\}\subseteq M$ and $|M|<\add(\Nwf)$. As $\add(\Nwf)\leq\bfrak$, there exists a real $d\in\omega^\omega$ dominating $M\cap\omega^\omega$. Define $b$, $h$ and $\Fwf=\{g_B\ /\ B\in\Bwf\}$ as in the proof of Theorem \ref{DestroyGradFragGaps}. Let $N\preceq H_\chi$ be such that $M\cup\{d\}\subseteq N$ and $|N|<\add(\Nwf)$. As $\add(\Nwf)\leq\blocfrak(b,h)$, we can find $\psi\in S(b,h)$ that $\in^*$-dominates all the reals in $\R_b\cap N$. Like in the proof of Theorem \ref{DestroyGradFragGaps}, we can construct $C$ that separates $\langle\Awf,\Bwf\rangle$.
\end{proof}

If in the proof of Theorem \ref{DestroyGradFragGaps} we consider a partition $\langle a_i\rangle_{i<\omega}$ such that $\langle |a_i|\rangle_{i<\omega}$ is bounded, then the resulting forcing $\Qor^h_b$ does not add anything new, which means that we can destroy $\Iwf$-gaps in this case by just adding dominating reals. Therefore, as a consequence of Lemma \ref{fragnotpseudo-tallchar}, it follows that

\begin{corollary}  \label{b(notpseudotall)=b}
   If $\Iwf$ is a nowhere tall fragmented ideal on $\omega$, then $\bfrak(\Iwf)=\bfrak$.
\end{corollary}

In the particular case of the gradually fragmented ideals in \ref{ExpFragId}(2), we even get a nice lower bound for the Rothberger number for each of these ideals.

\begin{lemma}\label{b(I_c(P))aboveLoc}
   Let $b,h\in\omega^\omega$ be functions converging to infinity such that $b\geq2$ and $h$ is non-decreasing. If $c\in\omega^\omega$ is such that $2\leq c$ and $h\leq^*c$ and $P=\{a_i\}_{i<\omega}$ is a partition of $\omega$ into non-empty finite sets such that $|a_i|\leq\log_2b(i)$ for all but finitely many $i<\omega$, then $\min\{\blocfrak(b,h),\bfrak\}\leq\bfrak(\Iwf_c(P))$. This also means that the forcing $\Dor\ast\Qnm^h_b$ destroys the Rothberger gaps of $\Iwf_c(P)$.
\end{lemma}

\begin{proof}
   As $b'\leq^* b$ implies $\blocfrak(b,h)\leq\blocfrak(b',h)$, it is enough to assume that $b(i)=2^{|a_i|}$ for all $i<\omega$. Note that, in the proof of Theorem \ref{DestroyGradFragGaps}, any sequence $\{N_l\}_{l<\omega}$ such that $c(i)\geq l$ for all $i\geq N_l$ serves for the purposes of that proof, so it can be defined in the ground model. In particular, choose such a sequence with the property $\forall_{i\in[N_l,N_{l+1})}(h(i)=l\leq c(i))$ for all but finitely many $l<\omega$. Define $h'(i)=l$ when $i\in[N_l,N_{l+1})$. By the argument of the same proof, $\min\{\blocfrak(b,h'),\bfrak\}\leq\bfrak(\Iwf_c(P))$ and, as $h'=^*h$, it is clear that $\blocfrak(b,h')=\blocfrak(b,h)$.
\end{proof}

\begin{remark}
   It is consistent that $\bfrak<\blocfrak(b,h)$ for every $b,h\in\omega^\omega$ such that $b>0$ and $h$ is a non-decreasing function that converges to infinity. This is because the forcing $\Qor^h_b$ satisfies that, for every name $\dot{g}$ of a real in $\omega^\omega$, there exists a real $f\in\omega^\omega$ in the ground model such that $\Vdash f'\not\leq^*\dot{g}$ for every real such that $f'\not\leq^*f$. The proof of this property is the same as the proof in \cite{miller} for the standard forcing that adds an eventually different real.

   By a book-keeping argument, if $\kappa\leq\lambda$ are uncountable regular cardinals and $\lambda^{<\kappa}=\lambda$, it is possible to perform a \emph{finite support iteration} (denoted by \emph{fsi} for short) $\Por_\lambda=\langle\Por_\alpha,\Qnm_\alpha\rangle_{\alpha<\lambda}$ alternating between suborders of $\Dor$ of size $<\kappa$ and posets of the form $\Qor_b^h$ such that, for any $\alpha<\lambda$ and $\Por_\alpha$-names $\dot{b}$ and $\dot{h}$, there is a $\beta\in(\alpha,\lambda)$ such that $\Qnm_\beta$ is $\Qor_{\dot{b}}^{\dot{h}}$, likewise for any $\Por_\alpha$-name for a suborder of $\Dor$ of size $<\kappa$. Known results for preservation properties in fsi (see, e.g., \cite{brendle} and \cite{mejia}) imply that, in the $\Por_\lambda$-extension, $\bfrak=\kappa$ and $\blocfrak(b,h)=\lambda$ for any $b,h$ as above (in the case that the cardinal exists).
\end{remark}

%
%

\section{Preservation properties}\label{SecPresProp}

We present some properties that help us to preserve the Rothberger number of a tall fragmented ideal small under certain forcing extensions. Actually, we present a new cardinal invariant that serves as upper bound for some of these Rothberger numbers and study a property for preserving this invariant small under generic extensions. Many ideas involved for this are inspired by \cite{kamoosuga}. At the end of this section, we discuss a property for preserving $\add(\Nwf)$ small in forcing extensions.

For this section, fix $\Iwf=\Iwf\langle a_i,\varphi_i\rangle_{i<\omega}$ a tall fragmented ideal, $2^{\bar{a}}:=\{\Pwf(a_i)\}_{i<\omega}$,  $\bar{L}:=\{L_n\}_{n<\omega}$ a partition of $\omega$ into infinite sets, $A_n:=\bigcup_{i\in L_n}a_i$ and $\Awf:=\{A_n\ / n<\omega\}$, which is also a partition of $\omega$ into infinite sets. Let $\Owf(\Iwf,\bar{L})$ be the collection of all the subsets of $\omega$ that are $\Iwf$-orthogonal with $\Awf$. In our applications, we will have that $\lim_{i\to+\infty}\varphi_i(a_i)=+\infty$ (a useful assumption for applying Theorem \ref{b(I)leqparam} and for saying something about the Rothberger number of $\Iwf$), but this is not a general requirement for the results in this section.

\begin{definition}\label{Defb(I)param}
   Let $\rho\in\omega^\omega$, $\rho>0$.
   \begin{enumerate}[(1)]
     \item For $\psi\in\prod_{i<\omega}\Pwf(\Pwf(a_i))$ and $Y\in\Owf(\Iwf,\bar{L})$, define
           \[\psi\backepsilon' Y\ \textrm{iff }\forall_{n<\omega}\exists^\infty_{i\in L_n}(\psi(i)\backepsilon Y\cap a_i).\]
     \item $\bfrak^\rho(\Iwf,\bar{L})$ is the least size of a subset $\Psi$ of $S(2^{\bar{a}},\rho)$ such that,
           for any $Y\in\Owf(\Iwf,\bar{L})$, there exists a $\psi\in\Psi$ such that $\psi\backepsilon' Y$.
   \end{enumerate}
   The reason for writing $\psi\backepsilon' Y$ instead of $Y\in'\psi$ lies in the nature of the problem of preserving the cardinal invariant $\bfrak^\rho(\Iwf,\bar{L})$. This is related to a general setting, explained in e.g. \cite[Sect. 6.4, 6.5]{barju} and \cite[Sect. 2]{mejia}, for preserving cardinal invariants of the form $\bfrak_\sqsubset$ in generic extensions, where $\sqsubset$ is an $F_\sigma$ relation of real numbers and $\bfrak_\sqsubset$ is defined as the smallest size of a set of reals that is $\sqsubset$-unbounded. In the case of this definition, we are interested in $\sqsubset$ to be the relation $\not\backepsilon^{\;\prime}$ defined on $S(2^{\bar{a}},\rho^{id_\omega})\times\Owf(\Iwf,\bar{L})$, so $\bfrak_\sqsubset=\bfrak^{\rho^{id_\omega}}(\Iwf,\bar{L})$. The properties presented in Definitions \ref{DefPresb(I)param} and \ref{DefPresAddN} are particular cases of one of the properties studied in the general setting.
\end{definition}

From now on, fix $\Ewf\subseteq\omega^\omega$ such that
\begin{enumerate}[(i)]
   \item For any $e\in\Ewf$, $e$ is non-decreasing, converges to infinity, $e\leq id_\omega$ and $id_\omega-e$ converges to infinity.
   \item If $e\in\Ewf$ then there exists an $e'\in\Ewf$ such that $e+1\leq^* e'$.
   \item If $\Cwf\subseteq\Ewf$ is countable, there exists an $e\in\Ewf$ that $\leq^*$-dominates all the reals in $\Cwf$.
\end{enumerate}
For $b,\rho\in\omega^\omega$, put $\widetilde{S}(b,\rho):=\bigcup_{e\in\Ewf}S(b,\rho^e)$. For $L\subseteq\omega$, denote by
$S(b,\rho)\frestr L:=\{\psi\frestr L\ /\ \psi\in S(b,\rho)\}$, likewise for $\widetilde{S}(b,\rho)\frestr L$.
For $m<\omega$, put $\Pwf_{m,i}(\Iwf):=\{x\subseteq a_i\ /\ \varphi_i(x)\leq m\}$, $S(\Iwf,L,m,\rho):=S(\{\Pwf_{m,i}(\Iwf)\}_{i<\omega},\rho)\frestr L$ and $\widetilde{S}(\Iwf,L,m,\rho):=\widetilde{S}(\{\Pwf_{m,i}(\Iwf)\}_{i<\omega},\rho)\frestr L$.  Note that $\widetilde{S}(b,1)=S(b,1)$ and $\widetilde{S}(\Iwf,L,m,1)=S(\Iwf,L,m,1)$.

\begin{theorem}\label{b(I)leqparam}
   If $\lim_{i\to+\infty}\varphi_i(a_i)/\rho(i)=+\infty$ then $\bfrak(\Iwf)\leq\bfrak^\rho(\Iwf,\bar{L})$.
\end{theorem}

\begin{proof}
  \begin{claim}\label{diagSlalom}
     Let $L\subseteq\omega$ be infinite, $A:=\bigcup_{i\in L}a_i$, $n<\omega$, $f\in\omega^\omega$ and $\{\psi_k\}_{k<\omega}$ a sequence of slaloms such that $\psi_k\in S(\Iwf,L,f(k),\rho)$. Then, there exists a $Z\in\Iwf\frestr A$ such that
     \begin{enumerate}[(i)]
       \item $\forall^\infty_{i\in L}(\varphi_i(Z\cap a_i)>n)$ and
       \item $\forall_{k<\omega}\forall^\infty_{i\in L}\forall_{x\in\psi_k(i)}(x\cap Z=\varnothing)$.
     \end{enumerate}
  \end{claim}
  \begin{proof}
    For $k<\omega$, put $m_k:=(k+1)\cdot\max_{j\leq k}\{f(j)\}$. Let $\{N_k\}_{k<\omega}$ be a strictly increasing sequence of natural numbers such that $\varphi_i(a_i)>n+m_k\cdot\rho(i)$ for all $i\geq N_k$. Now, by tallness, find $l<\omega$ as in Lemma \ref{tallfrag}(iii) applied to $m=n$. For $i\in L\cap[N_k,N_{k+1})$, as
    $\varphi_i(\bigcup_{j\leq k}\bigcup_{x\in\psi_j(i)}x)\leq m_k\cdot\rho(i)$, $a_i\menos\bigcup_{j\leq k}\bigcup_{x\in\psi_j(i)}x$ has submeasure bigger than $n$, so, by tallness, it contains a $z_i$ with submeasure in
    $(n,l]$. Therefore, $Z=\bigcup_{i\in L\cap[N_0,\omega)}z_i$ is as required.
  \end{proof}
  \begin{claim}\label{diagSlalom2}
     Let $L\subseteq\omega$ infinite, $A:=\bigcup_{i\in L}a_i$, $\psi\in S(2^{\bar{a}},\rho)\frestr L$ and $n<\omega$. Then, there exists a $Z_\psi\in\Iwf\frestr A$ such that
     \begin{enumerate}[(i)]
       \item $\forall^\infty_{i\in L}(\varphi_i(Z_\psi\cap a_i)>n)$ and
       \item $\forall_{k<\omega}\forall^\infty_{i\in L}\forall_{x\in\psi(i)}(\varphi_i(x)\leq k\imp x\cap Z_\psi=\varnothing)$.
     \end{enumerate}
  \end{claim}
  \begin{proof}
    For each $k<\omega$ put $\psi_k(i)=\{x\in\psi(i)\ / \varphi_i(x)\leq k\}$ and apply the previous claim with $f=id_\omega$.
  \end{proof}
  Now, let $\Psi\subseteq S(2^{\bar{a}},\rho)$ be a witness of $\bfrak^\rho(\Iwf,\bar{L})$. For each $\psi\in\Psi$ and $n<\omega$, let $Z_{\psi,n}\in\Iwf\frestr A_n$ be as in Claim \ref{diagSlalom2} applied to $L_n$, $A_n$, $\psi\frestr L_n$ and $n$. Put $Z_\psi:=\bigcup_{n<\omega}Z_{\psi,n}$, which is clearly in $\Owf(\Iwf,\bar{L})$. It is enough to prove that the orthogonal pair $\langle\Awf,\{Z_\psi\ / \ \psi\in\Psi\}\rangle$ is an $\Iwf$-gap. Let $X\in\Owf(\Iwf,\bar{L})$ and choose $\psi\in\Psi$ such that $\psi\backepsilon' X$. We show that, for any $n<\omega$ there is some $i<\omega$ such that $\varphi_i(a_i\cap Z_\psi\menos X)>n$. Choose $m<\omega$ such that $\bar{\varphi}(X\cap A_n)\leq m$, that is, $\varphi_i(X\cap a_i)\leq m$ for all $i\in L_n$. By Claim \ref{diagSlalom2}, choose a large enough $i\in L_n$ such that $\psi(i)\backepsilon X\cap a_i$, $\varphi_i(Z_\psi\cap a_i)>n$ and $X\cap a_i\cap Z_\psi=\varnothing$, so $\varphi_i(a_i\cap Z_\psi\menos X)>n$.
\end{proof}

\begin{definition}\label{DefPresb(I)param}
   Let $\rho\in\omega^\omega$, $\rho>0$, $\Por$ a poset and $\theta$ a cardinal number. Consider the following statement:
   \begin{description}
     \item[$(+^{<\theta}_{\Por,\bar{L},\Iwf,\rho})$] for every $m,n<\omega$ and $\dot{\psi}$ a $\Por$-name for a real in $\widetilde{S}(\Iwf,L_n,m,\rho)$, there exists a nonempty $\Psi\subseteq\widetilde{S}(\Iwf,L_n,m,\rho)$ of size $<\theta$ such that, for any $\psi''\in S(2^{\bar{a}},\rho^{id_\omega})$, if $\forall_{\psi'\in\Psi}\exists^\infty_{i\in L_n}(\psi''(i)\supseteq \psi'(i))$, then
         \[\Vdash\exists^\infty_{i\in L_n}(\psi''(i)\supseteq \dot{\psi}(i)).\]
   \end{description}
   $(+^1_{\Por,\bar{L},\Iwf,\rho})$ denotes $(+^{<2}_{\Por,\bar{L},\Iwf,\rho})$. Note that the previous property is simpler for the case $\rho=1$.
\end{definition}

This property serves for preserving the cardinal $\bfrak^{\rho^{id_\omega}}(\Iwf,\bar{L})$ small in generic extensions. From now on, fix an uncountable regular cardinal $\theta$. $\Psi''\subseteq S(2^{\bar{a}},\rho^{id_\omega})$ is said to be a \emph{$<\theta$-$\rho$-strong covering family (with respect to $\Iwf$ and $\bar{L}$)} if, for any $\Psi\subseteq\bigcup_{m,n<\omega}\widetilde{S}(\Iwf,L_n,m,\rho)$ of size $<\theta$, there exists a $\psi''\in\Psi''$ such that, for all $n<\omega$ and $\psi\in\Psi$ such that $\dom\psi=L_n$, $\exists^\infty_{i\in L_n}(\psi''(i)\supseteq\psi(i))$.

\begin{lemma}\label{Presb(I)paramfsi}
   \begin{enumerate}[(a)]
     \item $\bfrak^{\rho^{id_\omega}}(\Iwf,\bar{L})\leq|\Psi''|$ for any $<\theta$-$\rho$-strong covering family $\Psi''$.
     \item If $(+^{<\theta}_{\Por,\bar{L},\Iwf,\rho})$ holds, then $\Por$ preserves $<\theta$-$\rho$-strong covering families.
     \item If $\langle\Por_\alpha,\Qnm_\alpha\rangle_{\alpha<\theta}$ is a fsi of non-trivial ccc posets, then it adds a $<\theta$-$\rho$-strong covering family of size $\theta$ (of Cohen reals).
     \item $(+^{<\theta}_{\Por,\bar{L},\Iwf,\rho})$ is preserved by fsi of ccc posets.
   \end{enumerate}
\end{lemma}

\begin{proof}
   \begin{enumerate}[(a)]
     \item Let $Y\in\Owf(\Iwf,\bar{L})$ and, for $n<\omega$, put $\psi_n\in\widetilde{S}(\Iwf,L_n,\lceil\bar{\varphi}(Y\cap A_n)\rceil,\rho)$ where $\psi_n(i)=\{Y\cap a_i\}$ for all $i\in L_n$. As $\Psi''$ is a $<\theta$-$\rho$-strong covering family, there exists a $\psi''\in\Psi''$ such that, for all $n<\omega$, $\exists^\infty_{i\in L_n}(\psi''(i)\supseteq\psi_n(i))$. Therefore, $\psi''\backepsilon' Y$.
     \item Let $\Psi''$ be a $<\theta$-$\rho$-strong covering family. Let $\nu<\theta$ and, for $n,m<\omega$, let $\dot{\Psi}_{n,m}=\{\dot{\psi}_{n,m,\alpha}\ /\ \alpha<\nu\}$ be $\Por$-names for reals in $\widetilde{S}(\Iwf,L_n,m,\rho)$. For each $n,m<\omega$ and $\alpha<\nu$, let $\Psi'_{m,n,\alpha}\subseteq\widetilde{S}(\Iwf,L_n,m,\rho)$ be a witness of $(+^{<\theta}_{\Por,\bar{L},\Iwf,\rho})$ for $n$, $m$ and $\dot{\psi}_{n,m,\alpha}$, so it has size $<\theta$. As $\Psi''$ is a $<\theta$-$\rho$-strong covering family, there exists a $\psi''\in\Psi''$ such that $\exists^\infty_{i\in L_n}(\psi''(i)\supseteq\psi'(i))$ for all $\psi'\in\Psi'_{n,m,\alpha}$, $n,m<\omega$ and all $\alpha<\nu$. Thus, $\Por$ forces that $\exists^\infty_{i\in L_n}(\psi''(i)\supseteq\dot{\psi}_{n,m,\alpha})$.
     \item Consider Cohen forcing $\Cor=S_{<\omega}(2^{\bar{a}},\rho^{id_\omega})$ ordered by end extension. If $\dot{\psi}''$ is a $\Cor$-name for the Cohen generic real, it is clear that, for any $n<\omega$ and $\psi\in S(2^{\bar{a}},\rho^{id_\omega})\frestr L_n$, $\Cor$ forces that $\exists^\infty_{i\in L_n}(\dot{\psi}''(i)=\psi(i))$. Now, as $\langle\Por_\alpha,\Qnm_\alpha\rangle_{\alpha<\theta}$ adds Cohen reals at each limit step, we get that $\{\dot{\psi}''_{\alpha}\ /\ \alpha<\theta\textrm{ limit}\}$ is forced by $\Por_\theta$ to be a $<\theta$-$\rho$-strong covering family.
     \item This is a standard argument for preservation properties in fsi, see e.g. \cite[Thm. 6.4.12.2]{barju} and \cite{brendle}.
   \end{enumerate}
\end{proof}

Now, we explore some conditions for a poset to satisfy the property of Definition \ref{DefPresb(I)param}.

\begin{lemma}\label{Presb(I)paramcentered}
   Let $\nu<\theta$ be an infinite cardinal. If $\Por$ is a ccc $\nu$-centered poset, then $(+^{<\theta}_{\Por,\bar{L},\Iwf,\rho})$ holds.
\end{lemma}

\begin{proof}
   Let $\Por=\bigcup_{\alpha<\nu}P_\alpha$ where each $P_\alpha$ is centered. Fix $n,m<\omega$ and let $\dot{\psi}$ be a $\Por$-name for a real in $\widetilde{S}(\Iwf,L_n,m,\rho)$. By ccc-ness, we can find $e\in\Ewf$ such that $\dot{\psi}$ is forced to be in $S(\Iwf,L_n,m,\rho^e)$. For each $\alpha<\nu$ and $i\in L_n$, choose a $\psi'_{\alpha}(i)\subseteq\Pwf_{m,i}(\Iwf)$ of size $\leq\rho(i)^{e(i)}$ such that $\forall_{p\in P_\alpha}(p\not\Vdash\psi'_\alpha(i)\neq\dot{\psi}(i))$. This is possible by the centeredness of $P_\alpha$. Then, $\psi'_\alpha\in S(\Iwf,L_n,m,\rho^e)$. Put $\Psi:=\{\psi'_\alpha\ /\ \alpha<\nu\}$.

   $\Psi$ witnesses $(+^{<\theta}_{\Por,\bar{L},\Iwf,\rho})$ for $\dot{\psi}$. Indeed, let $\psi''\in S(2^{\bar{a}},\rho^{id_\omega})$ and assume that $\exists^\infty_{i\in L_n}(\psi''(i)$ $\supseteq\psi'_\alpha(i))$ for any $\alpha<\nu$. We show that $\Vdash\exists^\infty_{i\in L_n}(\psi''(i)\supseteq\dot{\psi}(i))$. Let $p\in\Por$ and $i_0\in\omega$ be arbitrary, choose $\alpha<\nu$ such that $p\in P_\alpha$ and also find $i>i_0$ in $L_n$ such that $\psi''(i)\supseteq\psi'_\alpha(i)$. By definition of $\psi'_\alpha(i)$, there exists a $q\leq p$ such that $q\Vdash\psi'_\alpha(i)=\dot{\psi}(i)$ so, clearly, $q\Vdash\psi''(i)\supseteq\dot{\psi}(i)$.
\end{proof}

We also want conditions that imply that a poset like in Definition \ref{DefLocPoset} satisfies a preservation property of Definition \ref{DefPresb(I)param}(2). The following notion is useful for this.

\begin{definition}[Kamo and Osuga {\cite[Def. 5]{kamoosuga}}]\label{Defparamlinked}
   Let $\pi,\rho\in\omega^\omega$. A poset $\Qor$ is \emph{$\langle\pi,\rho\rangle$-linked} if there is a sequence
   $\langle Q_{i,j}\rangle_{i<\omega,j<\rho(i)}$ of subsets of $\Qor$ such that
   \begin{enumerate}[(a)]
     \item $Q_{i,j}$ is $\pi(i)$-linked and
     \item for any $q\in\Qor$, $\forall^\infty_{i<\omega}\exists_{j<\rho(i)}(q\in Q_{i,j})$.
   \end{enumerate}
\end{definition}

It is clear that, if $\pi\geq^*2$, then any $\langle\pi,\rho\rangle$-linked poset is $\sigma$-linked.

\begin{lemma}\label{Presb(I)param linked}
   Let $\pi\in\omega^\omega$ be such that $\big|[\Pwf(a_i)]^{\leq\rho(i)^i}\big|\leq\pi(i)$ for all but finitely many $i<\omega$. If $\Qor$ is a $\langle\pi,\rho\rangle$-linked poset, then $(+^1_{\Qor,\bar{L},\Iwf,\rho})$ holds.
\end{lemma}

\begin{proof}
   Let $\langle Q_{i,j}\rangle_{i<\omega,j<\rho(i)}$ be a witness of the $\langle\pi,\rho\rangle$-linkedness of $\Qor$. Fix $n,m<\omega$ and let $\dot{\psi}$ be a $\Qor$-name for a real in $\widetilde{S}(\Iwf,L_n,m,\rho)$. By ccc-ness, find $e\in\Ewf$ such that $\dot{\psi}$ is a $\Qor$-name for a real in $S(\Iwf,L_n,m,\rho^e)$. For all but finitely many $i\in L_n$, for every $j<\rho(i)$, as $\pi(i)\geq\big|[\Pwf(a_i)]^{\leq\rho(i)^i}\big|$ and $Q_{i,j}$ is $\pi(i)$-linked, there is a $Y_{i,j}\subseteq\Pwf_{m,i}(\Iwf)$ of size $\leq\rho(i)^{e(i)}$ such that $\forall_{p\in Q_{i,j}}(p\not\Vdash\dot{\psi}(i)\neq Y_{i,j})$. Put $\psi'(i):=\bigcup_{j<\rho(i)}Y_{i,j}$.

   There exists an $e'\in\Ewf$ such that $e+1\leq^*e'$, so we may assume, by changing $\psi'(i)$ at finitely many $i$ if necessary, that $\psi'\in S(\Iwf,L_n,m,\rho^{e'})$. $\{\psi'\}$ witnesses $(+^1_{\Qor,\bar{L},\Iwf,\rho})$ for $\dot{\psi}$. Indeed, let $\psi''\in S(2^{\bar{a}},\rho^{id_\omega})$ such that $\exists^\infty_{i\in L_n}(\psi''(i)\supseteq\psi'(i))$. For $p\in\Qor$ and $i_0\in\omega$, choose an $i>i_0$ in $L_n$ and a $j<\rho(i)$ such that $\psi''(i)\supseteq\psi'(i)$ and $p\in Q_{i,j}$. Then, there exists a $q\leq p$ such that $q\Vdash Y_{i,j}=\dot{\psi}(i)$ so, clearly, $q\Vdash\psi''(i)\supseteq\psi'(i)\supseteq Y_{i,j}=\dot{\psi}(i)$.
\end{proof}

\begin{lemma}\label{LinkednessLocposet}
   Let $b,h\in\omega^\omega$ be non-decreasing functions with $b>0$ and $h$ converging to infinity. Let $\pi,\rho\in\omega^\omega$. If $\{m_k\}_{k<\omega}$ is a non-decreasing sequence of natural numbers that converges to infinity and, for all but finitely many $k<\omega$, $k\cdot\pi(k)\leq h(m_k)$ and
   $k\cdot|[b(m_k-1)]^{\leq k}|^{m_k}\leq\rho(k)$, then $\Qor^h_{b,\Fwf}$ is $\langle\pi,\rho\rangle$-linked for any $\Fwf\subseteq\R_b$.
\end{lemma}

\begin{proof}
    Choose $1<M<\omega$ such that, for any $k\geq M$, $k\cdot\pi(k)\leq h(m_k)$ and
   $k\cdot|[b(m_k-1)]^{\leq k}|^{m_k}\leq\rho(k)$. Find a non-decreasing sequence $\{n_k\}_{k<\omega}$ of natural numbers that converges to infinity such that, for all $k\geq M$, $n_k\leq k,m_k$ and $|S_{n_k}(b,h)|\leq k$. Let
   $S_k:=\{s\in S_{m_k}(b,h)\ /\ \forall_{i\in[n_k,m_k)}(|s(i)|\leq k)\}$ when $k\geq M$. Note that
   \[|S_k|\leq|S_{n_k}(b,h)|\cdot|\prod_{i\in[n_k,m_k)}[b(i)]^{\leq k}|\leq k\cdot|[b(m_k-1)]^{\leq k}|^{m_k}\leq\rho(k).\]
   For each $k\geq M$ and $s\in S_k$, put $Q_{k,s}:=\{(t,F)\in\Qor^h_{b,\Fwf}\ /\ t=s\textrm{\ and\ }|F|\cdot\pi(k)\leq h(m_k)\}$. It is clear that $Q_{k,s}$ is $\pi(k)$-linked for all $s\in S_k$. To conclude that $\Qor^h_{b,\Fwf}$ is $\langle\pi,\rho\rangle$-linked, we show that, given $(t,F)\in\Qor^h_{b,\Fwf}$, for all but finitely many $k$ we can extend $(t,F)$ to some condition in $Q_{k,s}$ for some $s\in S_k$. Choose $N<\omega$ such that $M,|F|\leq N$ and $|t|\leq n_N$. Extend $(t,F)$ to $(t',F)\in\Qor^h_{b,\Fwf}$ such that $|t'|=n_N$. Now, for all $k\geq N$, we can extend $(t',F)$ to $(s,F)\in\Qor^h_{b,\Fwf}$ such that $s\in S_k$ because $|F|\leq k$. For the same reason, we get $|F|\cdot\pi(k)\leq k\cdot\pi(k)\leq h(m_k)$ and, thus, $(s,F)\in Q_{k,s}$.
\end{proof}

We introduce the following property for preserving $\add(\Nwf)$ small in forcing extensions. This is a generalization of \cite[Def. 3.3]{jushe} that is useful for posets that satisfy some linkedness as presented in Definition \ref{Defparamlinked}.

\begin{definition}\label{DefPresAddN}
   Let $\bar{\Gwf}:=\{g_k\}_{k<\omega}$ be a sequence of functions in $\omega^\omega$ that converge to infinity. Put $S(\omega,\bar{\Gwf}):=\bigcup_{k<\omega}S(\omega,g_k)$ and, for a forcing notion $\Por$, define the following property:
   \begin{description}
     \item[$(+^{<\theta}_{\Por,\bar{\Gwf}})$] For any $k<\omega$ and any $\Por$-name $\dot{\psi}$ of a real in $S(\omega,g_k)$, there exists a nonempty $\Psi\subseteq S(\omega,\bar{\Gwf})$ of size $<\theta$ such that, for every $f\in\omega^\omega$, if $\forall_{\psi'\in\Psi}(f\notin^*\psi')$, then
         $\Vdash f\notin^*\dot{\psi}$.
   \end{description}
   $(+^{1}_{\Por,\bar{\Gwf}})$ denotes $(+^{<2}_{\Por,\bar{\Gwf}})$.

   A family $\Cwf\subseteq\omega^\omega$ is called \emph{$<\theta$-$\in^*$-$\bar{\Gwf}$-strongly unbounded} if, for any $\Psi\subseteq S(\omega,\bar{\Gwf})$ of size $<\theta$, there exists an $f\in\Cwf$ such that
   $f\notin^*\psi$ for any $\psi\in\Psi$.
\end{definition}

Note that, by Theorem \ref{BartcharAdd(N)} for $h=g_0$, if $\Cwf\subseteq\omega^\omega$ is $<\theta$-$\in^*$-$\bar{\Gwf}$-strongly unbounded, then $\add(\Nwf)\leq|\Cwf|$.

The following result is proved like Lemma \ref{Presb(I)paramfsi}. In fact, it is connected with results of \cite[Sect. 3]{jushe} (see also \cite[Subsect. 1.3]{brendle}) and the proofs are the same.

\begin{lemma}\label{PresAddNfsi}
   \begin{enumerate}[(a)]
      \item If $(+^{<\theta}_{\Por,\bar{\Gwf}})$ holds, then $\Por$ preserves $<\theta$-$\in^*$-$\bar{\Gwf}$-strongly unbounded families.
      \item If $\langle\Por_\alpha,\Qnm_\alpha\rangle_{\alpha<\theta}$ is a fsi of non-trivial ccc posets, then it adds a $<\theta$-$\in^*$-$\bar{\Gwf}$-strongly unbounded family of size $\theta$ (of Cohen reals).
      \item $(+^{<\theta}_{\Por,\bar{\Gwf}})$ is preserved by fsi of ccc posets.
   \end{enumerate}
\end{lemma}

We conclude this section by presenting some conditions saying when a poset satisfies the property of Definition \ref{DefPresAddN}.

\begin{lemma}\label{PresAddNcentered}
   Let $\nu<\theta$ be an infinite cardinal. If $\Por$ is $\nu$-centered, then $(+^{<\theta}_{\Por,\bar{\Gwf}})$ holds.
\end{lemma}

\begin{proof}
    The proof is like in \cite[Thm. 3.6]{jushe} and \cite[Lemma 6]{brendle}.
\end{proof}

\begin{lemma}\label{PresAddNlinked}
   Let $\pi,\rho\in\omega^\omega$ be such that $\lim_{k\to+\infty}\pi(k)=+\infty$ and assume $g\in\omega^\omega$ converges to infinity. Then, there is a $\leq^*$-increasing definable\footnote{That is, a continuous function $(\pi,\rho,g)\mapsto\bar{\mathcal{G}}$ (with Borel domain) can be constructed.} sequence $\bar{\Gwf}=\{g_k\}_{k<\omega}$ with $g_0=g$ and such that $(+^1_{\Qor,\bar{\Gwf}})$ holds for any $\langle\pi,\rho\rangle$-linked poset $\Qor$.
\end{lemma}

\begin{proof}
  This is a direct consequence of the following fact.
  \begin{claim}\label{NewfuncPresAddNlinked}
   Let $\{m_k\}_{k<\omega}$ be a strictly increasing sequence of natural numbers such that $g(k)<\pi(m_k)<\pi(m_{k+1})$ and let $g'\in\omega^\omega$ such that $g'(k)\geq g(k)\cdot\rho(m_k)$ for all but finitely many $k<\omega$. Then, if $\Qor$ is $\langle\pi,\rho\rangle$-linked and $\dot{\psi}$ is a $\Qor$-name for a real in $S(\omega,g)$, there exists a $\psi'\in S(\omega,g')$ such that, for any $f\in\omega^\omega$ such that $f\notin^*\psi'$, $\Vdash f\notin^*\dot{\psi}$.
  \end{claim}
  \begin{proof}
     Let $\langle Q_{k,j}\rangle_{k<\omega,j<\rho(k)}$ be a witness of the linkedness of $\Qor$. For any $k<\omega$ and $j<\rho(m_k)$, put $z_{k,j}:=\{l<\omega\ /\ \exists_{q\in Q_{m_k,j}}(q\Vdash l\in\dot{\psi}(k))\}$. As $Q_{m_k,j}$ is $\pi(m_k)$-linked and $g(k)<\pi(m_k)$, by linkedness it is clear that $|z_{k,j}|\leq g(k)$. Put $\psi'(k):=\bigcup_{j<\rho(m_k)}z_{k,j}$, so it is clear that $\psi'\in S(\omega,g')$. Let $f\in\omega^\omega$ such that $\exists^\infty_{k<\omega}(f(k)\notin\psi'(k))$ and we show that $\Vdash\exists^\infty_{k<\omega}(f(k)\notin\dot{\psi}(k))$. For $p\in\Qor$ and $k_0<\omega$, find $k>k_0$ and $j<m_k$ such that $f(k)\notin\psi'(k)$ and $p\in Q_{m_k,j}$. In particular, $f(k)\notin z_{k,j}$. By definition of $z_{k,j}$, $p\not\Vdash f(k)\in\dot{\psi}(k)$, so there is a $q\leq p$ such that $q\Vdash f(k)\notin\dot{\psi}(k)$.
  \end{proof}
\end{proof}

%
%

\section{Consistency results}\label{SecConsFrag}

In this section, we prove all our main consistency results for fragmented ideals (Theorem B). The first result says that it is consistent that the Rothberger numbers for all somewhere tall fragmented ideals are strictly less than $\bfrak$.

\begin{theorem}\label{allb(I)belowb}
  Let $\mu\leq\kappa$ be regular uncountable cardinals and $\lambda$ a cardinal such that $\lambda^{<\kappa}=\lambda$. Then, there exists a ccc poset that forces $\bfrak(\Iwf)\leq\mu$ for any somewhere tall fragmented ideal $\Iwf$, $\add(\Nwf)=\mu$, $\bfrak=\kappa$ and $\cfrak=\lambda$. In particular, this poset forces $\bfrak(\Iwf)=\add(\Nwf)=\mu$ for any somewhere tall gradually fragmented ideal $\Iwf$.
\end{theorem}

\begin{proof}
   We perform a fsi $\Por_\lambda=\langle\Por_\alpha,\Qnm_\alpha\rangle_{\alpha<\lambda}$ alternating between Cohen forcing $\Cor$, subalgebras of $\Loc^{id_\omega}$ of size $<\mu$ and subalgebras of $\Dor$ of size $<\kappa$ and, by a book-keeping argument, we make sure that all those subalgebras of the extension are used in the iteration (this is possible because $\lambda^{<\kappa}=\lambda$). By known techniques from \cite{brendle} (see also \cite[Sect.3]{mejia}), $\Por_\lambda$ forces $\add(\Nwf)=\mu$, $\bfrak=\kappa$ and $\cfrak=\lambda$.

   In $V$, fix $\bar{L}=\{L_n\}_{n<\omega}$ a partition of $\omega$ into infinite sets. Now, in $V_\lambda$, let $\Iwf=\Iwf\langle a_i,\varphi_i\rangle_{i<\omega}$ be a somewhere tall fragmented ideal and, by Remark \ref{RemGaps}(2), without loss of generality, assume that it is tall and $\lim_{n\to+\infty}\varphi_i(a_i)=+\infty$. As $\Iwf$ is represented by a real number, there exists $\alpha<\lambda$ such that $\langle a_i,\varphi_i\rangle_{i<\omega}\in V_\alpha$. By Lemma \ref{Presb(I)paramfsi}(c), there is a $<\mu$-$1$-strong covering family $\Psi''$ of size $\mu$ in $V_{\beta}$ where $\beta:=\alpha+\mu$ (ordinal sum). By Lemmas \ref{Presb(I)paramcentered} and \ref{Presb(I)paramfsi}, $\Por_{[\beta,\lambda)}=\Por_\lambda/\Por_\beta$ (the remaining part of the iteration from $\beta$) satisfies $(+^{<\mu}_{\cdot,\bar{L},\Iwf,1})$, so this $<\mu$-$1$-strong covering family $\Psi''$ is preserved in $V_\lambda$. By Theorem \ref{b(I)leqparam}, $\bfrak(\Iwf)\leq\bfrak^1(\Iwf,\bar{L})\leq|\Psi''|=\mu$.

   The last statement follows from Corollary \ref{b(gradfrag)aboveAddN}.
\end{proof}

The following shows that, for an ideal as in Example \ref{ExpFragId}(2), we can find a poset that puts its Rothberger number strictly between $\add(\Nwf)$ and $\bfrak$. In particular, this holds for the polynomial growth ideal $\Iwf_P$ as well.

\begin{theorem}\label{add(N)belowsomeb(I)}
   Let $\mu\leq\nu\leq\kappa$ be uncountable regular cardinals, $\lambda$ a cardinal such that $\lambda^{<\kappa}=\lambda$. Let $\Iwf=\Iwf_c(P)$ be a gradually fragmented ideal as in Example \ref{ExpFragId}(2) and assume it is non-trivial. Then, there exists a ccc poset that forces $\add(\Nwf)=\mu$, $\bfrak(\Iwf)=\nu$, $\bfrak=\kappa$ and $\cfrak=\lambda$.
\end{theorem}

\begin{proof}
   By Remark \ref{RemGaps}(2), it is enough to assume that $|a_i|\geq c(i)^i$ for every $i<\omega$ (this assumption is only used to prove $\bfrak(\Iwf)\leq\nu$ in the forcing extension defined below). Let $h=c$ and $b\in\omega^\omega$ any function such that $b(i)\geq2^{|a_i|}$ for any $i<\omega$. By Lemma \ref{LinkednessLocposet}, we can find $\pi,\rho\in\omega^\omega$ such that $\pi$ converges to infinity and $\Qor^h_{b,\Fwf}$ is $\langle\pi,\rho\rangle$-linked for any $\Fwf\subseteq\R_b$. Also, by Lemma \ref{PresAddNlinked}, find $\bar{\Gwf}$ such that $(+^1_{\Por,\bar{\Gwf}})$ holds for any $\langle\pi,\rho\rangle$-linked poset $\Por$.

   Perform a fsi $\Por_\lambda=\langle\Por_\alpha,\Qnm_\alpha\rangle_{\alpha<\lambda}$ alternating between Cohen forcing $\Cor$, subalgebras of $\Loc^{id_\omega}$ of size $<\mu$, $\Qor^h_{b,\Fwf}$ with $|\Fwf|<\nu$ and subalgebras of $\Dor$ of size $<\kappa$. By a book-keeping argument, we make sure that all such possible posets of the extension are used in the iteration. From the methods of \cite{brendle} (see also \cite[Sect. 3]{mejia}) it follows that $\add(\Nwf)\geq\mu$, $\bfrak=\kappa$ and $\cfrak=\lambda$ in $V_\lambda$.

   We prove that $\add(\Nwf)\leq\mu$ in $V_\lambda$. By Lemmas \ref{PresAddNcentered} and $\ref{PresAddNfsi}$ a $<\mu$-$\in^*$-$\bar{\Gwf}$-strongly unbounded family of size $\mu$ is added in $V_\mu$ and it is preserved in $V_\lambda$, so $\add(\Nwf)\leq\mu$.

   Now, in $V_\lambda$, $\nu\leq\blocfrak(b,h)$ (this implies $\nu\leq\bfrak(\Iwf)$ by Lemma \ref{b(I_c(P))aboveLoc}). Indeed, let $\Fwf\subseteq\R_b$ of size $<\nu$, so there is some $\alpha<\lambda$ such that $\Fwf\in V_\alpha$. Now, at some point of the remaining part of the iteration, the poset $\Qor^h_{b,\Fwf}$ is used to add a slalom $\psi\in S(b,h)$ that $\in^*$-dominates $\Fwf$.

   Finally, we prove that $\bfrak^1(\Iwf,\bar{L})\leq\nu$ is true in $V_\lambda$ (so $\bfrak(\Iwf)\leq\nu$ by Theorem \ref{b(I)leqparam}). Indeed, by Lemmas \ref{Presb(I)paramcentered} and \ref{Presb(I)paramfsi}, the iteration adds a $<\nu$-$1$-strong covering family of size $\nu$ in $V_\nu$ that is preserved in $V_\lambda$. Therefore, $\bfrak^1(\Iwf,\bar{L})\leq\nu$.
\end{proof}

The following result states that, no matter which (uncountable regular) values one wants to force for $\add(\Nwf)$ and $\bfrak$, it is consistent to find as many as possible gradually fragmented ideals that have pairwise different Rothberger numbers between $\add(\Nwf)$ and $\bfrak$.

\begin{theorem}\label{Consdiffb(I)}
   Let $\mu\leq\kappa$ be uncountable regular cardinals, $\delta\leq\kappa$ an ordinal, $\{\nu_\xi\}_{\xi<\delta}$ a non-decreasing sequence of regular cardinals in $[\mu,\kappa]$ and $\lambda$ a cardinal such that $\lambda^{<\kappa}=\lambda$. Then, there is a sequence $\{\Iwf_\xi\}_{\xi<\delta}$ of tall gradually fragmented ideals in the ground model and a ccc poset that forces $\add(\Nwf)=\mu$, $\bfrak=\kappa$, $\cfrak=\lambda$ and $\bfrak(\Iwf_\xi)=\nu_\xi$ for all $\xi<\delta$.
\end{theorem}

For the proof of this theorem, we use another characterization of the bounding number $\bfrak$. To fix some notation, define an elementary exponentiation operation $\sigma:\omega\times\omega\to\omega$ given by $\sigma(n,0)=1$ and $\sigma(n,m+1)=n^{\sigma(n,m)}$. Put $\rho:\omega\to\omega$ such that $\rho(0)=2$ and
$\rho(i+1)=\sigma(\rho(i),i+3)$. For a function $x\in\omega^\omega$ define, by recursion on $k<\omega$, $x^{[0]}=x$ and $x^{[k+1]}=2^{\rho^2\cdot x^{[k]}}$. Now, let
\[\R^\rho:=\{ x\in\omega^\omega\ /\ \forall_{k<\omega}\forall^\infty_{i<\omega}(x^{[k]}(i)\leq\rho(i+1))\}.\]

\begin{lemma}\label{operR^rho}
   Let $x\in\R^\rho$. Then,
   \begin{enumerate}[(a)]
     \item $id_\omega\in\R^\rho$.
     \item The functions $2^x$, $x^{id_\omega\cdot\rho^{id_\omega}}$, $id_\omega\cdot x$ and $y$ defined as
           $y(i)=|[x(i)]^{\leq\rho(i)^i}|$ are in $\R^\rho$.
     \item $z\in \R^\rho$ where $z(i)=\max\{x(j)\}_{j\leq i}$.
     \item $\forall^\infty_{i<\omega}(i\cdot|[x(i-1)]^{\leq i}|^i\leq\rho(i))$.
   \end{enumerate}
\end{lemma}

\begin{proof}
  \begin{enumerate}[(a)]
     \item It is enough to show that $id_\omega^{[k]}(i)\leq\sigma(\rho(i),2k+1)$ for all $i<\omega$ and $k<\omega$ by induction on $k$. The case $k=0$ is   clear. For the induction step,
         \begin{multline*}
             id_\omega^{[k+1]}(i)=2^{id_\omega^{[k]}(i)\cdot\rho(i)^2}\leq\rho(i)^{\sigma(\rho(i),2k+1)\cdot
                    \rho(i)^{\rho(i)}}\\
                    \leq\rho(i)^{\sigma(\rho(i),2k+2)}=\sigma(\rho(i),2k+3).
         \end{multline*}
     \item Clear because $2^x\leq^*x^{[1]}$, $x^{id_\omega\cdot\rho^{id_\omega}}\leq^* x^{[2]}$, $id_\omega\cdot x\leq^* x^{[1]}$ and $y\leq^* x^{[1]}$.
     \item Let $N_k$ be minimal such that $\forall_{i\geq N_k}(x^{[k]}(i)\leq\rho(i+1))$. As the sequence $\{x^{[k]}(i)\}_{k<\omega}$ is strictly increasing for each $i<\omega$, $\{N_k\}_{k<\omega}$ is non-decreasing and converges to infinity. Let $\{k_j\}_{j<\omega}$ be a strictly increasing sequence of natural numbers such that $\rho(N_{k_j}+1)$ is bigger than $x^{[j]}(i)$ for all $i<N_j$. Then, $z^{[j]}(i)\leq\rho(i+1)$ for all $i\geq N_{k_j}$.
     \item Note that $(i+1)|[x(i)]^{\leq(i+1)}|^{i+1}\leq2^{x(i)\cdot(i+1)\cdot 2}\leq x^{[1]}(i)$ for all but finitely many $i$ such that $x(i)\neq 0$. The case $x(i)=0$ is straightforward.
  \end{enumerate}
\end{proof}

To proceed with the proof of Theorem \ref{Consdiffb(I)}, we first need to see that $\bfrak$ is the least size of a $\leq^*$-unbounded family in $\R^\rho$. But we can prove a more general result instead. Fix $g,H\in\omega^\omega$ such that $H$ is strictly increasing and $id_\omega<H$. Define $\R^g_H:=\{x\in\omega^\omega\ /\ \forall_{k<\omega}(x^{[k]}\leq^* g)\}$ where $x^{[0]}=x$ and $x^{[k+1]}=H\circ x^{[k]}$. With the particular case $g(i)=\rho(i+1)$ and $H=2^{\rho^2\cdot id_\omega}$, what we need is just a consequence of the following.

\begin{lemma}\label{b(R^rho)}
   Assume that $\R^g_H\neq\varnothing$. Then, $\bfrak$ is the least size of a $\leq^*$-unbounded family in $\R^g_H$.
\end{lemma}

\begin{proof}
   For $x\in\R^g_H$ and $k<\omega$, let $N^x_k$ be the minimal $N<\omega$ such that $\forall_{i\geq N}(x^{[k]}(i)\leq g(i))$. $\{N^x_k\}_{k<\omega}$ is non-decreasing and converges to infinity because $H>id_\omega$. Consider the function $H'$ of natural numbers such that
   $H'(m)$ is the maximal $n<\omega$ such that $H(n)\leq m$. Note that the domain of $H'$ is $[H(0),\omega)$ and that $H(n)\leq m$ iff $H'(m)$ is defined and $n\leq H'(m)$. Also, $H'(m)<m$ for all $m\geq H(0)$. Define, for $k<\omega$, the function
   $C_k$ on a subset of $\omega$ by $C_0(i)=g(i)$ and $C_{k+1}(i)=H'(C_k(i))$.
   \begin{claim}\label{claimR^rho}
      Let $x\in\omega^\omega$, $i,k<\omega$. Then, $x^{[k]}(i)\leq g(i)$ iff $C_k(i)$ exists and $x(i)\leq C_k(i)$.
   \end{claim}
   \begin{proof}
      Fix $i<\omega$ and $0<M<\omega$. Define, when possible, $C^M_k$ for $k<\omega$ such that $C^M_0=M$ and $C^M_{k+1}=H'(C^M_k)$. It is enough to prove, by induction on $k$ that, for all $0<M<\omega$, $x^{[k]}(i)\leq M$ iff $C^M_k$ exists and $x(i)\leq C^M_k$ (our claim is the particular case $M=g(i)$).
      The case $k=0$ is trivial, so we proceed to prove the inductive step. $x^{[k+1]}(i)\leq M$ is equivalent to $x^{[k]}(i)\leq C_1^M$ which is equivalent, by induction hypothesis, to the existence of $C^{C^M_1}_k$ and $x(i)\leq C^{C^M_1}_k=C^M_{k+1}$.
   \end{proof}
   Note that, as there is come $c\in\R^g_H$, the functions $C_k$ are defined for all but finitely many natural numbers and they converge to infinity. Indeed, by the previous claim, $\forall_{i\geq N^{c}_{k+M}}(c(i)+M\leq C_k(i))$.
   Now, consider $W$ as the set of non-decreasing functions $z\in\omega^\omega$ such that $\forall_{k<\omega}(z(k)\geq N^{c}_k)$. It is clear that $\bfrak$ is the least size of a $\leq^*$-unbounded family in $W$.

   Define the function $F:W\to\R^g_H$ such that, $F(z)=F_z:\omega\to\omega$, $F_z(i)=C_k(i)$ when $i\in[z(k),z(k+1))$ (we don't care about the values below $z(0)$). Claim \ref{claimR^rho} guarantees that $F_z\in\R^g_H$. Also, let $F':\R^g_H\to\omega^\omega$, $F'(x)=F'_x$ such that $F'_x(i)=N^x_i$. The lemma follows from the fact that, for any $x\in\R^g_H$ and $z\in W$,
   \begin{enumerate}[(i)]
     \item $F_z\leq^* x$ implies $z\leq^* F'_x$, and
     \item $F'_x\leq^*z$ implies $x\leq^* F_z$.
   \end{enumerate}
   To prove (i), assume that there is a $\bar{k}<\omega$ such that $\forall_{i\geq z(\bar{k})}(F_z(i)\leq x(i))$. Let $k'<\omega$ be minimal such that $z(\bar{k})<z(k')$ and prove that $z(k)\leq N^x_k$ for all $k\geq k'$. By contradiction, assume that there is a minimal $k\geq k'$ such that $N^x_k<z(k)$, so there exists an $i\in[z(k-1),z(k))$ with $i\geq N^x_k$. Then, $C_{k-1}(i)=F_z(i)\leq x(i)$ and $x(i)\leq C_k(i)$ (by Claim \ref{claimR^rho}). But $C_k(i)=H'(C_{k-1}(i)) < C_{k-1}(i)$, a contradiction.

   For (ii), assume that there is a $\bar{k}<\omega$ such that $N^x_k\leq z(k)$ for all $k\geq\bar{k}$. If $i\geq z(\bar{k})$, we can find a $k\geq\bar{k}$ such that $i\in[z(k),z(k+1))$, so $x(i)\leq C_k(i)=F_z(i)$ because $i\geq z(k)\geq N_k^x$.
\end{proof}

\begin{proof}[Proof of Theorem \ref{Consdiffb(I)}]
   Without loss of generality, we may assume that $\bfrak=\kappa$ in the ground model $V$.
   Construct, for $\xi<\delta$, functions $h_\xi,a_\xi,b_\xi,\pi_\xi\in\omega^\omega$ such that
   \begin{enumerate}[(a)]
      \item $h_0=id_\omega^2$ and, for $\xi>0$, $h_\xi$ is non-decreasing, converges to infinity and, for $\eta<\xi$, $id_\omega\cdot\pi_\eta\leq^* h_\xi$
      \item $a_\xi>0$ is non-decreasing, converges to infinity and $h_\xi^{id_\omega\cdot\rho^{id_\omega}}\leq^*a_\xi$,
      \item $b_\xi$ and $\pi_\xi$ are defined as $b_\xi=2^{a_\xi}$ and $\pi_\xi(i)=|[b_\xi(i)]^{\leq\rho(i)^i}|$, and
      \item $\forall^\infty_{i<\omega}(i\cdot|[b_\xi(i-1)]^{\leq i}|^i\leq\rho(i))$.
   \end{enumerate}
   We can construct all those functions in $\R^\rho$. To see this, fix $\xi<\delta$ and assume that we have all these functions for $\eta<\xi$. By Lemma \ref{operR^rho}, $id_\omega\cdot\pi_\eta\in\R^\rho$ for all $\eta<\xi$, so there exists a non-decreasing function $h_\xi\in\R^\rho$ bounding them by Lemma \ref{b(R^rho)}. Put $a_\xi:=\max\{h_\xi^{id_\omega\cdot\rho^{id_\omega}},1\}$, which is in $\R^\rho$. Clearly, $b_\xi,\pi_\xi\in\R^\rho$ and (d) is true.

   For each $\xi<\delta$, let $P_\xi=\{a_{\xi,i}\}_{i<\omega}$ be the interval partition of $\omega$ such that $|a_{\xi,i}|=a_\xi(i)$, define $c_\xi(i)=\max\{h_\xi(i),2\}$ and let $\Iwf_\xi:=\Iwf_{c_\xi}(P_\xi)$ (see Example \ref{ExpFragId}(2)). Perform a fsi $\Por_{(3+\delta)\cdot\lambda}:=\langle\Por_\alpha,\Qnm_\alpha\rangle_{\alpha<(3+\delta)\cdot\lambda}$ such that, for $\gamma<\lambda$,
   \begin{enumerate}[(i)]
      \item If $\alpha=(3+\delta)\cdot\gamma$, let $\Qnm_\alpha$ be a $\Por_\alpha$-name for Cohen forcing,
      \item if $\alpha=(3+\delta)\cdot\gamma+1$, let $\Qnm_\alpha$ be a $\Por_\alpha$-name for a subalgebra of $\Loc^{id_\omega}$ of size $<\mu$,
      \item if $\alpha=(3+\delta)\cdot\gamma+2$, let $\Qnm_\alpha$ be a $\Por_\alpha$-name for a subalgebra of $\Dor$ of size $<\kappa$,
      \item if $\alpha=(3+\delta)\cdot\gamma+3+\xi$ for $\xi<\delta$, let $\Qnm_\alpha=\Qor^{h_\xi}_{b_\xi,\dot{\Fwf}_\alpha}$ where $\dot{\Fwf}_\alpha$ is a $\Por_\alpha$-name of a subset of $\R_{b_\xi}$ of size $<\nu_\xi$.
   \end{enumerate}
   By a book-keeping argument, we make sure to use all such posets of the extension in the iteration.  Choose $\bar{L}=\{L_n\}_{n<\omega}$ any partition of $\omega$ into infinite sets.
   \begin{claim}\label{claimconsdifb(I)presb(I)param}
      For every $\xi<\delta$ and $\alpha<(3+\delta)\cdot\lambda$, $\Por_\alpha$ forces $(+^{<\nu_\xi}_{\Qnm_\alpha,\bar{L},\Iwf_\xi,\rho})$
   \end{claim}
   \begin{proof}
      Let $\alpha=(3+\delta)\cdot\gamma+\xi''$ for some $\gamma<\lambda$ and $\xi''<3 + \delta$. Step in $V_\alpha$. If $\xi''\leq2$ then $(+^{<\mu}_{\Qor_\alpha,\bar{L},\Iwf,\rho})$ holds by Lemma \ref{Presb(I)paramcentered}; if $\xi''=3+\xi'$ for some $\xi'<\delta$ then, if $\xi'\leq\xi$, as $|\Qor_\alpha|<\nu_{\xi'}\leq\nu_\xi$, $(+^{<\nu_\xi}_{\Qor_\alpha,\bar{L},\Iwf_\xi,\rho})$ holds by Lemma \ref{Presb(I)paramcentered}; if $\xi<\xi'$, by (a),(d) and Lemma \ref{LinkednessLocposet} (with $m_k=k$), $\Qor_\alpha=\Qor^{h_{\xi'}}_{b_{\xi'},\Fwf_\alpha}$ is $\langle\pi_\xi,\rho\rangle$-linked, thus, by (c) and Lemma \ref{Presb(I)param linked}, $(+^{1}_{\Qor_\alpha,\bar{L},\Iwf_\xi,\rho})$ holds.
   \end{proof}
   \begin{claim}\label{claimconsdifb(I)presaddN}
      In $V$, there is a sequence $\bar{\Gwf}=\{g_k\}_{k<\omega}$ of reals in $\omega^\omega$ that converges to infinity such that $\Por_\alpha$ forces $(+^{<\mu}_{\Qnm_\alpha,\bar{\Gwf}})$.
   \end{claim}
   \begin{proof}
       By Lemma \ref{PresAddNlinked}, find
       $\bar{\Gwf}=\{g_k\}_{k<\omega}$ such that $(+^{1}_{\Qor,\bar{\Gwf}})$ holds for $\langle id_\omega,\rho\rangle$-linked posets $\Qor$. Now, step in $V_\alpha$. If $\alpha=(3+\delta)\cdot\gamma+\xi'$ for some $\xi'<3+\delta$, when $\xi'\leq 2$ then $(+^{<\mu}_{\Qor,\bar{\Gwf}})$ holds by Lemma \ref{PresAddNcentered}; else, if $\xi'=3+\xi$ for some $\xi<\delta$, the claim holds because $\Qor_\alpha$ is $\langle id_\omega,\rho\rangle$-linked (see the proof of the previous claim).
   \end{proof}
   It is known that, in $V_{(3+\delta)\cdot\lambda}$, $\bfrak=\kappa$ and $\cfrak=\lambda$. Also, $\add(\Nwf)\geq\mu$ because of the small subalgebras of $\Loc^{id_\omega}$ used in the iteration. By the same argument as in Theorem \ref{add(N)belowsomeb(I)}, we get, for $\xi<\delta$, $\nu_\xi\leq\blocfrak(b_\xi,h_\xi)$ so, by Lemma \ref{b(I_c(P))aboveLoc}, $\nu_\xi\leq\bfrak(\Iwf_\xi)$.

   To see $\add(\Nwf)\leq\mu$ note that, by Claim \ref{claimconsdifb(I)presaddN} and Lemma \ref{PresAddNfsi}, in $V_\mu$ we add a $<\mu$-$\in^*$-strongly unbounded family of size $\mu$ that is preserved in $V_{(3+\delta)\cdot\lambda}$, so $\add(\Nwf)\leq\mu$. Likewise, by Claim \ref{claimconsdifb(I)presb(I)param} and Lemma \ref{Presb(I)paramfsi}, for $\xi<\delta$, we add in $V_{\nu_\xi}$ a $<\nu_\xi$-$\rho$-strong covering family (with respect to $\Iwf_\xi$ and $\bar{L}$) of size $\nu_\xi$ that is preserved in $V_{(3+\delta)\cdot\lambda}$, so $\bfrak^{\rho^{id_\omega}}(\Iwf_\xi,\bar{L})\leq\nu_\xi$. But, by Theorem \ref{b(I)leqparam}, as $\lim_{i\to+\infty}\varphi_i(a_{\xi,i})/(\rho(i)^i)=+\infty$ by (b), we get $\bfrak(\Iwf_\xi)\leq\bfrak^{\rho^{id_\omega}}(\Iwf_\xi,\bar{L})\leq\nu_\xi$.
\end{proof}

To obtain (consistently) continuum many pairwise different Rothberger numbers, it is necessary that the continuum is a weakly inaccessible cardinal. Indeed, let $\{\Iwf_\xi\}_{\xi<\cfrak}$ be a sequence of ideals such that the numbers $\bfrak(\Iwf_\xi)$ are pairwise different. As there are continuum many and all of them are $\leq\bfrak$, we obtain $\cfrak=\bfrak$, so $\cfrak$ is regular. Also, as there are $\cfrak$-many different cardinals below $\cfrak$, $\cfrak$ has to be a limit cardinal. Likewise, the existence of $\bfrak$-many different Rothberger numbers implies that $\bfrak$ is weakly inaccessible.

\begin{corollary}
   Assume that $\lambda$ is a weakly inaccessible cardinal such that $\lambda^{<\lambda}=\lambda$, and let $\mu<\lambda$ be a regular cardinal. For any collection of pairwise different regular cardinals $\{\nu_\xi\}_{\xi<\lambda}\subseteq[\mu,\lambda]$, there exist tall gradually fragmented ideals $\Iwf_\xi$ for $\xi<\lambda$ and a ccc poset that forces $\add(\Nwf)=\mu$, $\bfrak=\cfrak=\lambda$ and $\bfrak(\Iwf_\xi)=\nu_\xi$ for any $\xi<\lambda$.
\end{corollary}

%
%

\section{Questions}\label{SecQ}

We have seen that $\bfrak(\Iwf) \geq \add(\Nwf)$ for all gradually fragmented ideals $\Iwf$ (Corollary~\ref{b(gradfrag)aboveAddN}) while for a large class of fragmented not gradually fragmented ideals $\Iwf$ we have $\bfrak(\Iwf)=\aleph_1$ (Theorems~\ref{b(EDfin)=aleph1}, \ref{b(unifsubmeas)=aleph1}, and~\ref{b(meas)=aleph1}). This gives rise to the following

\begin{conjecture}[dichotomy conjecture, Hru\v s\'ak]  \label{dichconj}
For fragmented ideals $\Iwf$,
\begin{itemize}
\item $\bfrak(\Iwf) = \aleph_1$ if $\Iwf$ is not gradually fragmented,
\item $\bfrak(\Iwf) \geq \add(\Nwf)$ if $\Iwf$ is gradually fragmented.
\end{itemize}
\end{conjecture}

Even if this is not true, we do believe there is a dichotomy in the sense that for a large class of definable ideals, one of the two alternatives in~\ref{dichconj} holds, and that there is a natural combinatorial characterization saying which way it goes.

\begin{problem}[dichotomy conjecture, general version]
\begin{enumerate}[(a)]
   \item For every fragmented ideal $\Iwf$, either $\bfrak(\Iwf) = \aleph_1$ or $\bfrak(\Iwf) \geq \add(\Nwf)$.
   \item For every $F_\sigma$ ideal $\Iwf$, either $\bfrak(\Iwf) = \aleph_1$ or $\bfrak(\Iwf) \geq \add(\Nwf)$.
   \item For every analytic ideal $\Iwf$, either $\bfrak(\Iwf) = \aleph_1$ or $\bfrak(\Iwf) \geq \add(\Nwf)$.
   \item In (a), (b), or (c), give a combinatorial characterization of the ideals satisfying either case of the dichotomy.
\end{enumerate}
\end{problem}

By Theorem~\ref{add(N)belowsomeb(I)}, for a large class of gradually fragmented ideals $\Iwf$, $\bfrak(\Iwf) > \add(\Nwf)$ is consistent, and we do not know of a $\thzfc$ example of a definable ideal $\Iwf$ such that $\bfrak(\Iwf) = \add(\Nwf)$.

\begin{question}
\begin{enumerate}[(a)]
   \item Is there a (gradually) fragmented ideal $\Iwf$ with $\bfrak(\Iwf) = \add(\Nwf)$? Or is it consistent that $\bfrak(\Iwf) > \add(\Nwf)$ for \emph{all} non-trivial gradually fragmented ideals?
   \item Is there an analytic ideal $\Iwf$ such that $\bfrak(\Iwf) = \add(\Nwf)$?
\end{enumerate}
\end{question}

Our results can be seen as saying that there are two fundamentally distinct linear gaps in many quotients by fragmented ideals: one is of type $(\omega, \bfrak)$ (by Todor\v cevi\'c's results~\cite{todorcevic}) while the other is either of type $(\omega, \omega_1)$ (e.g. in $\Pwf(\omega)/\mathcal{ED}_\textrm{fin}$) or of some type $(\omega, \bfrak(\Iwf))$ with $\add(\Nwf) \leq \bfrak (\Iwf) \leq \bfrak$ (e.g. in $\Pwf (\omega) / \Iwf_P$). A natural question is whether there can even be a third type of linear gap in (some of) these quotients. Notice that while there may be linear $(\omega,\kappa)$-gaps in $\Pwf (\omega) / \mathrm{Fin}$ for several $\kappa$, all of these gaps ``look similar". Namely, by Rothberger's classical result~\cite{rothb}, there is a linear $(\omega,\kappa)$-gap in $\Pwf (\omega) / \mathrm{Fin}$ iff there is a well-ordered unbounded sequence in $(\omega^\omega, \leq^*)$ of length $\kappa$.

\begin{problem}
Characterize those $\kappa$ for which there is a linear $(\omega,\kappa)$-gap in $\Pwf (\omega) / \mathcal{ED}_\textrm{fin}$ (in $\Pwf (\omega) / \Iwf_L$, in $\Pwf (\omega) / \Iwf_P$).
\end{problem}

Choosing countably many of the ideals $\Iwf_n = \Iwf_{c_n} (P_n) = \Iwf \langle a_{n,i}, \varphi_{n,i} \rangle_{i < \omega}$ of the proof of Theorem~\ref{Consdiffb(I)}, letting $h : \omega \times \omega \to \omega$ be a bijection, and defining an ideal $\Iwf$ on $\omega$ by stipulating $x \in \Iwf$ iff there is $k$ such that $\varphi_{n,i} ( h^{-1} [x \cap h[a_{n,i} \times \{n\}] ]) \leq k$ for all $n$ and $i$, one obtains a fragmented ideal such that countable many definable cardinals, namely all the cardinals $\bfrak (\Iwf_n)$, belong to the gap spectrum of $\Iwf$, by Remark~\ref{RemGaps} (2). We do not know whether one can have more definable cardinals in the spectrum.

\begin{question}
Is there a fragmented ideal $\Iwf$ (an $F_\sigma$ ideal) such that, for a family $\bfrak_f$ of consistently (mutually) distinct definable cardinals (uniformly) parametrized by $f \in \omega^\omega$, there is a linear Rothberger gap of type $(\omega, \bfrak_f)$ in $\Pwf (\omega) / \Iwf$?
\end{question}

The definability assumption about the $\bfrak_f$ is arguably a bit vague, but in view of the comment about $\Pwf (\omega) / \mathrm{Fin}$ made above, it is necessary to avoid trivialities.

In the context of Todor\v cevi\'c's results about preservation of gaps by Baire embeddings \cite{todorcevic}, it would be interesting to investigate the existence of embeddings $F:\Pwf(\omega)/\Jwf\to\Pwf(\omega)/\Iwf$ preserving gaps where $\Jwf$ and $\Iwf$ are fragmented ideals (or analytic ideals in general). Note that, in the proof of Corollary \ref{b(LinGr)=aleph1}, we constructed such a continuous embedding $F:\Pwf(\omega)/\mathcal{ED}_\textrm{fin}\to\Pwf(\omega)/\Iwf$ where $\Iwf$ is as in Example \ref{ExpFragId}(4), but we still do not know whether similar embeddings can be constructed in other cases. For example, one may ask whether (some) gradually fragmented ideals can be embedded into fragmented not gradually fragmented ideals like $\mathcal{ED}_\textrm{fin}$ in a gap-preserving way, or whether there can be such embeddings between the distinct ideals of the proof of Theorem~\ref{Consdiffb(I)}.

%
%

\end{document}